\documentclass[11pt, reqno]{amsart}
\usepackage{a4wide}
\usepackage{wrapfig}
\usepackage{amsmath}
\usepackage{amsthm}
\usepackage{mathrsfs}
\usepackage{color}
\usepackage{xcolor}
\usepackage{graphicx}
\usepackage{amssymb}
\usepackage{ dsfont }
\usepackage{url}
\usepackage{multicol}
\usepackage{tikz}
\usepackage{amsmath}
\usepackage{fancyhdr}
\usetikzlibrary{matrix}
\usetikzlibrary{arrows}
\usepackage{subfig}
\usepackage{hyperref}
\usepackage{listings}

\usepackage{esint}

\usepackage{enumitem}
\usepackage{varwidth}

\allowdisplaybreaks

\title[Orbital and asymptotic stability of a train of peakons for the Novikov equation]{Orbital and asymptotic stability of a train of peakons for the Novikov equation}

\author[J.M. Palacios]{Jos\'e Manuel Palacios}
\address{Institut Denis Poisson, Universit\'e de Tours, Universit\'e d'Orleans, CNRS, Parc Grandmont 37200, Tours, France}
\email{jose.palacios@etu.univ-tours.fr}

\newcommand{\be}{\begin{equation}}
\newcommand{\ee}{\end{equation}}
\newcommand{\bp}{\begin{proof}}
\newcommand{\ep}{\end{proof}}
\newcommand{\bel}{\begin{equation}\label}
\newcommand{\eeq}{\end{equation}}
\newcommand{\bea}{\begin{eqnarray}}
\newcommand{\eea}{\end{eqnarray}}
\newcommand{\bee}{\begin{eqnarray*}}
\newcommand{\eee}{\end{eqnarray*}}
\newcommand{\ben}{\begin{enumerate}}
\newcommand{\een}{\end{enumerate}}

\newcommand{\R}{\mathbb{R}}

\newcommand{\N}{\mathbb{N}}

\newcommand{\sgn}{\operatorname{sgn}}

\newtheorem{thm}{Theorem}[section]
\newtheorem{cor}[thm]{Corollary}
\newtheorem{lem}[thm]{Lemma}

\newtheorem{defn}[thm]{Definition}

\theoremstyle{remark}
\newtheorem{rem}{Remark}[section]

\definecolor{codegreen}{rgb}{0,0.6,0}
\definecolor{codegray}{rgb}{0.5,0.5,0.5}
\definecolor{codepurple}{rgb}{0.58,0,0.82}
\definecolor{backcolour}{rgb}{0.95,0.95,0.92}

\lstdefinestyle{mystyle}{
	backgroundcolor=\color{backcolour},   
	commentstyle=\color{codegreen},
	keywordstyle=\color{magenta},
	numberstyle=\tiny\color{codegray},
	stringstyle=\color{codepurple},
	basicstyle=\footnotesize,
	breakatwhitespace=false,         
	breaklines=true,                 
	captionpos=b,                    
	keepspaces=true,                 
	numbers=left,                    
	numbersep=5pt,                  
	showspaces=false,                
	showstringspaces=false,
	showtabs=false,                  
	tabsize=2
}
\numberwithin{equation}{section}

\usepackage{pgfplots}
\pgfplotsset{compat=newest}



\theoremstyle{definition}

\numberwithin{ej}{section}

\begin{document}





\renewcommand{\sectionmark}[1]{\markright{\thesection.\ #1}}
\renewcommand{\headrulewidth}{0.5pt}
\renewcommand{\footrulewidth}{0.5pt}
\begin{abstract}
The Novikov equation is an integrable Camassa-Holm type equation with cubic nonlinearity. One of the most important features of this equation is the existence of peakon and multi-peakon solutions, i.e.  peaked traveling waves behaving as solitons. 
This paper aims to prove both, the orbital and asymptotic stability of peakon trains solutions, i.e. multi-peakon solutions such that their initial configuration is increasingly ordered. Furthermore, we give an improvement of the orbital stability of a single peakon so that we can drop the non-negativity hypothesis on the momentum density. The same result also holds for the orbital stability for peakon trains, i.e. in this latter case we can also avoid assuming non-negativity of the initial momentum density. 
 Finally, as a corollary of these results together with some asymptotic formulas for the position and momenta vectors for multi-peakon solutions, we obtain the orbital and asymptotic stability for initially not well-ordered multipeakons.
\end{abstract}

\maketitle 

\section{Introduction}

\subsection{The model} This paper is concerned with the Novikov equation \begin{align}\label{novikov_eq}
u_t-u_{txx}+4u^2u_x=3uu_xu_{xx}+u^2u_{xxx}, \qquad t\in\R,\,x\in\R,
\end{align}
where $u(t)$ is a real-valued function. This equation was derived by Novikov \cite{No} in a symmetry classification of nonlocal partial differential equations with cubic nonlinearity. By using the perturbative symmetry approach \cite{MiNo}, which yields necessary conditions for a PDE to admit infinitely many symmetries, Novikov  was able to isolate equation \eqref{novikov_eq} and derive its first few symmetries. Later, he was able to find an associated scalar Lax-pair, proving the integrability of the equation. Moreover, Hone and Wang have recently found a matrix Lax-pair representation of the Novikov equation, more specifically, they have shown that equation \eqref{novikov_eq} arises as the zero curvature equation $F_t-G_x+[F,G]=0$ which is the compatibility condition for the linear system \cite{HoWa} \begin{align*}
\Psi_x=F(y,\lambda)\Psi \quad \hbox{and} \quad \Psi_t=G(y,\lambda)\Psi,
\end{align*}
where $y=u-u_{xx}$ and the matrices $F$ and $G$ are defined by \begin{align*}
F=\left(\begin{matrix}
0 & \lambda y & 1 \\ 0 & 0 & \lambda y \\ 1 & 0 & 0
\end{matrix}\right), \quad G= \left(\begin{matrix}
\tfrac{1}{3\lambda^2}-uu_x & \tfrac{1}{\lambda}u_x-\lambda u^2y & u_x^2 \\ \tfrac{1}{\lambda}u & -\tfrac{2}{3\lambda^2} & -\tfrac{1}{\lambda}u_x-\lambda u^2y 
\\ -u^2 & \tfrac{1}{\lambda}u & \tfrac{1}{3\lambda^2}+uu_x
\end{matrix}\right).
\end{align*}
Moreover, by using this matrix Lax-pair representation, Hone and Wang showed how the Novikov equation is related by a reciprocal transformation to a negative flow in the Sawada-Kotera hierarchy. 

\medskip

The Novikov equation possesses infinitely many conservation laws, among which, the most important ones are given by
\begin{align}\label{cons_e}
E(u):=\int_\R\left(u^2(t,x)+u_x^2(t,x)\right)dx \quad \hbox{and}\quad F(u):=\int \Big(u^4+2u^2u_x^2-\dfrac{1}{3}u_x^4\Big)dx.
\end{align}
Solutions of \eqref{novikov_eq} are known to satisfy several symmetry properties: shifts in space and time, i.e. if $u(t,x)$ is a solution to equation \eqref{novikov_eq} then so is $u(t+t_0,x+x_0)$, as well as space-time invertion, i.e. if $u(t,x)$ is a solution of \eqref{novikov_eq}, then $u(-t,-x)$ is another solution.

\medskip

One of the most important features of the Novikov equations is the existence of \emph{peakon} and \emph{antipeakon} solutions \cite{HoWa} which are peaked traveling waves with a discontinuous derivative at the crest. In this case, for any $c>0$ they are explicitly given by \begin{align}\label{peakons_nov_def}
\pm\varphi_{c}(x-ct)=\pm\sqrt{c}\varphi(x-ct):=\pm\sqrt{c}e^{-\vert x-ct\vert}.
\end{align}
Moreover, the Novikov equation also exhibit multi-peakons-antipeakons solutions. More precisely, for any given natural number $n\in\N$, let us denote by $\vec{q}=(q_1,...,q_n)$ and $\vec{p}=(p_1,...,p_n)$ the position and momenta vectors respectively. Then, the $n$-peaked traveling wave solution on the line is given by $
u(t,x)=\sum_{i=1}^n p_i(t)\exp(-\vert x-q_i(t)\vert)$, where $p_i$ and $q_i$ satisfy the following system of $2n$-differential equations
\begin{align}\label{multipeak}
\begin{cases}
\dfrac{dq_i}{dt}=u^2\big(t,q_i(t)\big)=\displaystyle\sum_{j,k=1}^np_jp_ke^{-\vert q_i-q_j\vert-\vert q_i-q_k\vert},
\\ \displaystyle\dfrac{dp_i}{dt}=-p_i(t)u\big(t,q_i(t)\big)u_x\big(t,q_i(t)\big)=p_i\sum_{j,k=1}^np_jp_k\sgn(q_i-q_j)e^{-\vert q_i-q_j\vert-\vert q_i-q_k\vert}.\end{cases}
\end{align}
There exist some similar expressions for periodic peakons and multipeakons solutions but we do not intend to deepen into this direction in this work. On the other hand, equation \eqref{novikov_eq} can be rewritten in a \textit{compact} form in terms of its \emph{momentum density} as
\begin{align}\label{nov_eq_y}
y_t+u^2y_x+3uu_xy=0, \quad \hbox{where} \quad y:=u-u_{xx},
\end{align}
which can be regarded as a cubic nonlinear generalization of the celebrated Camassa-Holm (CH) equation \cite{CH,FuFo}, \begin{align}\label{CH}
u_t-u_{txx}=uu_{xxx}+2u_xu_{xx}-3uu_x,\quad \hbox{equivalently} \quad y_t+uy_x+2u_xy=0,
\end{align}
or the Degasperis-Procesi (DP) equation \cite{DP}, \begin{align}\label{DP}
u_t-u_{txx}=uu_{xxx}+3u_xu_{xx}-4uu_x, \quad \hbox{equivalently} \quad y_t+uy_x+3u_xy=0.
\end{align}
It is worth noticing that the last three equations in terms of their momentum densities correspond to transport equations for $y(t)$. As a consequence, initial data with signed initial momentum density give rise to solutions with the same property. This is one of the key points to prove that smooth and sufficiently fast decaying initial data with signed initial momentum density give rise to global solutions.

\medskip

Regarding the CH and the DP equations, both can be derived as a model for the propagation of unidirectional shallow water waves over a flat bottom by writing the Green-Naghdi equations in Lie-Poisson Hamiltonian form and then making an asymptotic expansion which keeps the Hamiltonian structure \cite{AlLa,CH,CoLa,Jo}. Moreover, both of them can be written in Hamiltonian form as \[
\partial_tE'(u)=-\partial_xF'(u),
\]
where for the Camassa-Holm equation $E(u)$ and $F(u)$ are given by \[
E_{CH}(u):=\int u^2+u_x^2 \quad \hbox{and} \quad F_{CH}(u):=\int u^3+uu_x^2
\]
while for the Degasperis-Procesi equation they are given by \[
E_{DP}(u):=\int yv=\int 5v^2+4v_x^2+v_{xx}^2 \quad \hbox{and}\quad F_{DP}(u):=\int u^3,
\]
where $v:=(4-\partial_x^2)^{-1}u$. Moreover, both of them belongs to the so-called $b$-family introduced by Degasperis, Holm and Hones in \cite{DeHoHo}, \[
u_t-u_{txx}=bu_xu_{xx}+uu_{xxx}-(b+1)uu_x.
\]
In \cite{MiNo} it was shown that the $b$-family corresponds to an integrable equation only when $b=2,3$, which corresponds exactly to the CH and the DP equations respectively.

\medskip

On the other hand, the Novikov equation, as well as the CH and the DP equations, can also be written in a nonlocal form in the following way. From now on we shall denote by $p(x)$ the fundamental solution of $1-\partial_x^2$ in $\R$, that is $p:=\tfrac{1}{2}e^{-\vert x\vert}$. Then, we can rewrite \eqref{novikov_eq} as 
\begin{align}\label{nov_eq_2}
u_t+u^2u_x=-p*\left(3uu_xu_{xx}+2u_x^3+3u^2u_x\right),
\end{align}
which can be understood as a nonlocal perturbation of Burgers-type equations \[
u_t+\tfrac{1}{3}(u^3)_x=0,
\]
or more generally as a nonlinear nonlocal transport equation. This latter fact has many implications, for instance, from the blow-up criteria for transport equations we obtain that singularities are caused by the focusing of characteristics. It is worth noticing that, in order to give peakons and multi-peakons a precise meaning as (weak) solutions of the Novikov equation, it is necessary to rewrite equation \eqref{novikov_eq} in the non-local form as in \eqref{nov_eq_2}. In fact, due to their non-smoothness they can not be understood as strong solutions of the equation\footnote{Another way of defining peakons and multi-peakons as weak solutions of the Novikov equations is by rewriting \eqref{novikov_eq} in a derivative form.}. 

\medskip

At this point it is clear that the Novikov equation shares many of its remarkable analytic properties with both the CH and the DP equations, as the existence of a Lax-pair, the completely integrability and the bi-Hamiltonian structure \cite{DP,HoWa}, but also all of them exhibit both existence of peaked traveling waves as well as the phenomenon of wave breaking \cite{CH,ChGuLiQu,CoLa,DP,No}. This latter one means that the wave profile remains bounded while its slope becomes unbounded. As the authors explain in \cite{ChGuLiQu}, understanding the wave-breaking mechanism not only presents fundamental importance from a mathematical point of view but also a great physical interest since it would help to provide a key-mechanism for localizing energy in conservative systems by
forming one or several small-scale spots. Finally, we remark that, unlike the Novikov equation, peakon solutions for the CH and the DP equations have a slightly different form, which is given by \begin{align}\label{ch_dp_peakon_def}
\widetilde{\varphi}_c(x-ct)=c\varphi(x-ct):=ce^{-\vert x-x_0-ct\vert}, \qquad c\in\R\setminus\{0\},\  x_0\in\R.
\end{align}
It is worth noticing that in sharp contrast with the Novikov equation, CH and DP peakons can move in both directions, left and right, just by changing the sign of $c$, while all Novikov peakons and anti-peakons move to the right.

\medskip

About the stability of these peaked solitary waves, the first proof of orbital stability was given in the Camassa-Holm case for $H^1$-perturbations assuming that their associated momentum density defines a non-negative Radon measure \cite{CoMo2}. The orbital stability for perturbations in the whole energy space $H^1(\R)$ was proved by a direct approach by Constantin and Strauss in \cite{CoSt} (see also \cite{LiLi} for a proof in the Degasperis-Procesi case). Later, following the ideas in \cite{CoSt,LiLi} Liu et al. proved the orbital stability for Novikov's peakon solutions under the additional assumption of non-negative initial momentum density \cite{LiLiQu}. In this work we shall prove that we can drop this latter hypothesis (see Theorem \ref{MT4} below).

\medskip
 
From a physical point of view, all of these peakon solutions \eqref{peakons_nov_def} and \eqref{ch_dp_peakon_def} reveal some similarities to the well-known Stokes waves of greatest height, i.e. traveling waves of maximum possible amplitude that are solutions to the governing equations for irrotational water waves \cite{Co,To}. These traveling waves (Stokes waves) are smooth everywhere except at the crest, where the lateral tangents differ. Then, it is important from both a physical and a mathematical point of view to study these types of solutions.


%
%

\subsection{Initial data space}

Before stating our main results we need to introduce the functional spaces where our initial data shall belong. Following the ideas of \cite{CoMo,EM1,EM2,Mo,Pa} we define  
\[
Y:=\big\{u\in H^1(\R): \ u-u_{xx}\in\mathcal{M}_b\big\},
\]
where $\mathcal{M}_b$ denotes the space of Radon measures  with finite total variation on $\R$. Moreover, from now on we shall denote by $Y_+$ the subspace defined by $Y_+:=\{u\in Y: \ u-u_{xx}\in\mathcal{M}_{b}^+\}$, where $\mathcal{M}_{b}^+$ denotes the space of non-negative finite Radon measures on $\R$. A crucial remark in what follows is that, for any function $v\in C_0^\infty(\R)$ we have
\begin{align}\label{positive_mom_1}
v(x)&=\dfrac{1}{2}\int_{-\infty}^x e^{x'-x}(v-v_{xx})(x')dx'+\dfrac{1}{2}\int_x^\infty e^{x-x'}(v-v_{xx})(x')dx'
\end{align}
and \begin{align}\label{positive_mom_2}
v_x(x)&=-\dfrac{1}{2}\int_{-\infty}^x e^{x'-x}(v-v_{xx})(x')dx'+\dfrac{1}{2}\int_x^\infty e^{x-x'}(v-v_{xx})(x')dx'
\end{align}
Therefore, if $v-v_{xx}\geq 0$ on $\R$ we conclude that $\vert v_x\vert\leq v$. Thus, by density of $C_0^\infty(\R)$ in $Y$, we deduce the same property for functions $v\in Y_+$.

\begin{rem}
We recall the following standard estimate which shall be useful in the sequel:
\[
\Vert u\Vert_{W^{1,1}}=\Vert p*(u- u_{xx})\Vert_{W^{1,1}}\lesssim \Vert u-u_{xx}\Vert_{\mathcal{M}},
\]
and hence it also holds that \[
\Vert u_{xx}\Vert_{\mathcal{M}}\leq \Vert u\Vert_{L^1}+\Vert u-u_{xx}\Vert_{\mathcal{M}}.
\]
Thus, we have $
Y(\R)\hookrightarrow \left\{u\in W^{1,1}(\R): \, u_x\in \mathrm{BV}(\R)\right\}$,
where $\mathrm{BV}(\R)$ denotes the space of functions with bounded variation. 
\end{rem}

\subsection{Main results}
As we mentioned before, in this work we intend to address both, the orbital and asymptotic stability problems for a train of peakons.

\subsubsection{Orbital stability in the energy space}

Our first result is an improvement of the orbital stability property for the single peakon solution. Indeed, by some slight improvements and modifications of the proof in \cite{LiLiQu}, we shall show that the sign assumption on the momentum density is artificial, and hence it can be removed. This is rather an observation regarding the fact that the proof follows the one in \cite{CoSt} for the Camassa-Holm equation.
\begin{thm}[Orbital stability of peakons in the energy space]\label{MT4}
Let $c>0$ be fixed. There exists $ 0<\varepsilon^\star\ll \min\{1,\sqrt{c}\}$ small enough such that if \[
u\in L^\infty((-T,T), H^1(\R)\cap W^{1,4}(\R)),
\]
is a solution to the Novikov equation \eqref{nov_eq_2} emanating from an initial data $u_0\in H^1(\R)\cap W^{1,4}(\R)$ satisfying \begin{align}\label{initial_cond_hyp_peakon}
\left\Vert u_0-\varphi_{c}\right\Vert_{H^1}+\Vert u_{0,x}-\varphi_c'\Vert_{L^4}\leq \varepsilon^4 ,\quad \hbox{for some}\quad 0<\varepsilon<\varepsilon^\star,
\end{align}
such that $E(\cdot)$ and $F(\cdot)$ are conserved along the trajectory, then, the following estimate holds: \[
\quad \sup_{t\in[-T,T]}\Vert u(t)-\varphi_c(\cdot-\xi(t))\Vert_{H^1}\leq 2c^{3/8}\big(4+\textbf{c}\big)\varepsilon, \ \quad \textbf{c}:=\max\{1,c^{3/8}\},
\]
where $\xi(t)\in\R$ is any point where the function $u(t,\cdot)$ attains its maximum.
\end{thm}

\begin{rem}
It is worth noticing that, in contrast to the Camassa-Holm case, continuity with respect to time of the solution is not needed here. Specifically, we only need the quantities $E(\cdot)$ and $ F(\cdot) $ to be conserved. Indeed,  for any $ v\in H^1\cap W^{1,4}$, it holds that \[
\Vert v-\varphi\Vert_{H^1} \to 0 \quad\hbox{as}\quad  \vert E(v)-E(\varphi)\vert+\vert F(v)-F(\varphi)\vert \to 0, 
\]
whereas the analogous result $\Vert v-\widetilde\varphi\Vert_{H^1} \to 0 $ as $ \vert E_{CH}(v)-E_{CH}(\widetilde\varphi)\vert +\vert F_{CH}(v)-F_{CH}(\widetilde\varphi)\vert \to 0 $ only holds if we additionally assume that $ v$ belongs to some $ L^\infty$-neighborhood of $ \widetilde\varphi$. 
\end{rem}

Once we have proved the orbital stability of a single peakon in some functional space we may consider the orbital stability problem for a train of peakons under the same hypothesis. In this regard, we obtain the analogous result to the last theorem for peakon train solutions of the Novikov equation.

\begin{thm}[Orbital stability of a train of peakons in the energy space]\label{MT2}
Let $c_1,...c_n$ be $n$ real numbers such that $0<c_1<...<c_n$. There exists $\varepsilon^\star>0$ small enough, $L_0>0$ and an universal constant $C>0$ such that if for some $0<T\leq +\infty$, \[
u\in C([0,T),H^1(\R))\cap L^\infty([0,T),W^{1,4}(\R))
\]
is a solution to the Novikov equation \eqref{nov_eq_2} emanating from an initial data $u_0\in H^1(\R)\cap W^{1,4}(\R)$ such that $E(\cdot)$ and $F(\cdot)$ are conserved along the trajectory and satisfying \begin{align}\label{initial_cond_hyp_train}
\left\Vert u_0-\sum_{i=1}^n \varphi_{c_i}(\cdot-z_i)\right\Vert_{H^1}+\left\Vert u_{0,x}-\sum_{i=1}^n \varphi_{c_i}'(\cdot-z_i)\right\Vert_{L^4}\leq \varepsilon^4 ,\quad \hbox{with}\quad 0<\varepsilon<\varepsilon^\star,
\end{align}
for some array of numbers $\{z_i\}_{i=1}^n\subset\R$ with $z_{i+1}-z_i\geq L$ where $L\geq L_0$, then the following holds: There exist $C^1$ functions $x_1(t),...,x_n(t):[0,T)\to\R$ such that \begin{align}\label{orb_concl}
\sup_{t\in[0,T)}\left\Vert u(t,\cdot)-\sum_{i=1}^n\varphi_{c_i}(\cdot-x_i(t))\right\Vert_{H^1}\lesssim \varepsilon+L^{-1/8}.
\end{align}
Additionally, we have $x_{i+1}(t)-x_i(t)>\tfrac{L}{2}$ for all $t\in[0,T)$.
\end{thm}

\begin{rem}
Notice that the local existence assumptions of Theorems \ref{MT4} and \ref{MT2} are satisfied, in particular, for initial data in $Y(\R)$ (see Theorem \ref{theorem_lwp} below).
\end{rem}
On the other hand, in \cite{HoLuSz} Hones et al. studied  the asymptotic behavior of multipeakon solutions in the case when no antipeakons are allow. In particular, the limits of $p_i(t)$ and $\dot{q}_i(t)$ in \eqref{multipeak} as $t$ goes to $+\infty$ are determined. As a corollary of the previous theorem together with the study made in \cite{HoLuSz} we obtain the orbital stability of the whole manifold \[
\mathcal{N}:=\left\{v(x)=\sum_{i=1}^n p_ie^{-\vert x-q_i\vert}: \ p_1,...,p_n\in \R_+, \ q_1<...<q_n\right\}.
\]
In concrete, we have the following result:
\begin{cor}[Orbital stability of not-well ordered multi-peakons]\label{cor_MT2}
Let $p_1^0,...,p_n^0$ be $n$ positive real numbers and $q_1^0<...<q_n^0$. For any $\alpha>0$ and any $\delta>0$ there exists $\varepsilon>0$ such that for any initial data $u_0\in Y_+(\R)$ satisfying \begin{align}\label{hyp_cor_orb}
\left\Vert u_0-\sum_{i=1}^np_i^0\exp\big(-\vert \cdot-q_i^0\vert\big)\right\Vert_{H^1}\leq \varepsilon \quad \hbox{with}\quad \Vert y_0\Vert_{\mathcal{M}}\leq \alpha,
\end{align}
then the following holds: For all time $t\in\R$ we have \begin{align}\label{first_part_cor}
\inf_{\vec{q}\in\R^n,\,\vec{p}\in\R_+^n}\left\Vert u(t,\cdot)-\sum_{i=1}^np_i\exp\big(-\vert \cdot-q_i\vert\big)\right\Vert_{H^1}\leq \delta.
\end{align}
Moreover, there exists $T>0$ sufficiently large such that \[
\hbox{for all }\, t\geq T, \quad \inf_{\vec{q}\in\mathcal{G}}\left\Vert u(t,\cdot)-\sum_{i=1}^n\lambda_i\exp\big(-\vert \cdot-q_i\vert\big)\right\Vert_{H^1}\leq \delta,
\]
and \[
\hbox{for all }\, t\leq- T, \quad \inf_{\vec{q}\in\mathcal{G}}\left\Vert u(t,\cdot)-\sum_{i=1}^n\lambda_{n+1-i}\exp\big(-\vert \cdot-q_i\vert\big)\right\Vert_{H^1}\leq \delta.
\]
where $\mathcal{G}:=\{\vec{q}\in\R^n: \ q_1<...<q_n\}$ and the parameters $0<\lambda_1<...<\lambda_n$ are the square roots of the eigenvalues of the matrix $TPEP$, where: \[
P:=\mathrm{diag}\left(p_1^0,...,p_n^0\right), \quad E:=\left(e^{-\vert q_i-q_j\vert}\right)_{i,j=1}^n \ \hbox{ and } \,\ T:=\left(1+\mathrm{sgn}(j-k)\right)_{i,j=1}^n.
\]
\end{cor}

\subsubsection{Asymptotic stability results}
About the asymptotic stability of peaked traveling waves, the case of a single peakon have recently been addressed and proved by the author by a proof based in a rigidity property of the Novikov equation (see \cite{Pa} Theorems $1.2$ and $1.3$).
\begin{thm}[\cite{Pa}]\label{AS_single_peakon}
Let $c>0$ be fixed. There exists an universal constant $1\gg\varepsilon^\star>0$ such that for any $\beta\in(0,c)$ and initial data $u_0\in Y_+$ satisfying 
\begin{align}\label{AS_smallness_hip}
\Vert u_0-\varphi_c\Vert_{H^1}\leq \varepsilon^\star\Big(\tfrac{\beta}{c}\Big)^8,
\end{align}
then the following property holds: There exists $c^*>0$ with $\vert c-c^*\vert\ll c$ and a $C^1$ function $x:\R\to\R$ satisfying $\dot{x}(t)\to c^*$ as $t\to+\infty$
\[
u(t,\cdot+x(t))\rightharpoonup \varphi_{c^*} \ \hbox{ in }\ H^1(\R).
\]
where\footnote{By this we mean that $u\in C(\R,H^1(\R))$ with $y\in C_{ti}(\R,\mathcal{M}_b(\R))$. See definition \ref{def_c_ti} below.} $u\in C_{ti}(\R,Y_+)$ is the global weak solution to equation \eqref{nov_eq_2} associated to $u_0$. Moreover, for any $z\in\R$ the following strong convergence holds
\begin{align}\label{AS_strong_h1_conv_peakon_mthm}
\lim_{t\to+\infty}\Vert u(t)-\varphi_{c^*}(\cdot-x(t))\Vert_{H^1((-\infty,z)\cup(\beta t,+\infty))}=0.
\end{align}
\end{thm}

As we mentioned before, the main ingredient in the proof of Theorem \ref{AS_single_peakon} is a rigidity property of the Novikov equation ensuring that every $H^1$-almost localized solution (c.f. Definition $1.1$ in \cite{Pa}) to equation \eqref{nov_eq_2} is actually a peakon. This latter property has been proved by introducing a new Lyapunov functional not related to the (not conserved) momentum of the equation. The main result of the present work is the asymptotic stability for peakon train solutions.

\begin{thm}[Asymptotic stability of a train of peakons]\label{MT5}
Let $c_1,...c_n$ be $n$ positive real number satisfying $c_1<...<c_n$ and $\beta\in(0,\tfrac{c_1}{4})$. There exists $L_0>0$ sufficiently large and $\varepsilon^\star>0$ small enough such that if a solution $u\in C_{ti}(\R,Y_+(\R))$ to the Novikov equation associated to some initial data $u_0\in Y_+(\R)$ satisfies \begin{align}\label{smallness_hyp_train_peakons}
\left\Vert u_0-\sum_{i=1}^n\varphi_{c_i}(\cdot-z_i)\right\Vert_{H^1}\leq \varepsilon^4, \quad \hbox{with } \ 0<\varepsilon<\varepsilon^\star,
\end{align}
for some $\{z_i\}_{i=1}^n\subset\R$ satisfying $z_{i+1}-z_i\geq L$ for some $L\geq L_0$ then the following holds: There exists $n$ positive real numbers $c_1^\star<...<c_n^\star$ and $C^1$ functions $x_1,...,x_n:\R\to\R$ such that for all $i=1,...,n,$ \begin{align*}
\dot{x}^\star_i(t)\to c_i^\star \ \hbox{ as } \ t\to+\infty \ \ \hbox{ and } \ \ u\big(t,\cdot+x_i(t)\big)\rightharpoonup \varphi_{c_i^\star} \ \hbox{ in } \ H^1 \ \hbox{ as } \ t\to+\infty.
\end{align*}
Moreover, for any $z\in\R$ the following strong convergence holds: \begin{align}\label{MT_5_conclusions}
\lim_{t\to+\infty}\left\Vert u(t)-\sum_{i=1}^n\varphi_{c_i^\star}\big(\cdot-x_i^\star(t)\big)\right\Vert_{H^1(\mathcal{A}_t)}=0, \quad \hbox{with } \ \mathcal{A}_t:=(-\infty,z)\cup(\beta t,+\infty)
\end{align}
\end{thm}

Finally, gathering the latter theorem with the asymptotic obtained by Hones et al in \cite{HoLuSz} we obtain the following asymptotic stability result for the whole manifold $\mathcal{N}$. 

\begin{cor}[Asymptotic stability of not-well ordered multi-peakons]\label{cor_MT5}
Let $p_1^0,...,p_n^0$ be $n$ positive real numbers, $q_1^0,...,q_n^0$ any sequence of real numbers satisfying $q_1^0<...<q_n^0$ and let us consider $0<\lambda_1<...<\lambda_n$ defined as in Corollary \ref{cor_MT2}. For any $\alpha>0$, any $\delta>0$ there exists $\varepsilon>0$ sufficiently small such that if $u_0\in Y_+(\R)$ satisfies \[
\left\Vert u_0-\sum_{i=1}^np_i^0\exp\big(-\vert \cdot-q_i^0\vert\big)\right\Vert_{H^1}\leq \varepsilon \quad \hbox{with}\quad \Vert y_0\Vert_{\mathcal{M}}\leq \alpha,
\]
then the following holds: There exists $0<p_1<...<p_n$ and $C^1$ functions $q_1,...,q_n:\R\to\R$ satisfying \[
\quad \vert p_i-\lambda_i\vert\leq\delta \quad \hbox{and}\quad \lim_{t\to+\infty}\dot{q}_i(t)=p_i^2, \quad 
\hbox{for all } i=1,...,n,
\]
such that \[
u(t)-\sum_{i=1}^n\varphi_{p_i}(\cdot-q_i(t))\xrightarrow{t\to+\infty} 0 \ \hbox{ in } \ H^1\left(\big(\tfrac{\lambda_1}{4}t,+\infty\big)\right).
\]
\end{cor}

\begin{rem}
Notice that due to the fact that the Novikov equation \eqref{nov_eq_2} is invariant under the transformation $u(t,x)\mapsto -u(t,x)$, we also deduce the orbital and asymptotic stability result for a train of antipeakon profiles with perturbations in the class of $H^1$ functions with momentum density belonging to $\mathcal{M}_b^-(\R)$. 
\end{rem}

\subsection{Organization of this paper}
This paper is organized as follows. In Section \ref{preliminaries} we introduce some  definitions and state a series of results needed in our analysis, for instance, the local and global well-posedness results in the class of solutions we shall work with. In section \ref{sec_MT4} we prove the orbital stability result for a single peakon solution. In section \ref{sec_MT2}, following the ideas of the latter section we prove the orbital stability of a train of peakons. Finally, in section \ref{sec_MT5} we prove the asymptotic stability of peakon trains for $Y_+(\R)$ perturbations.

\medskip

\section{Preliminaries}\label{preliminaries}

\subsection{Preliminaries and definitions}

In order to make regularization arguments, in the sequel we shall need the following family of functions: Let $\{\rho_n\}_{n\in\N}$ be a mollifiers family definied by \begin{align}\label{def_rho}
\rho_n(x):=n\left(\int_\R\rho(\xi)d\xi\right)^{-1}\rho(nx), \quad \hbox{ where } \quad \rho(x):=\begin{cases}
e^{\frac{1}{x^2-1}} & \hbox{for } \vert x\vert<1
\\ 0 & \hbox{for } \vert x\vert\geq 1.
\end{cases}
\end{align}
It worth to notice that for any $n\in\N$ we have $\Vert \rho_n\Vert_{L^1}=1$. On the other hand, from now on we shall denote by $C_b(\R)$ the set of bounded continuous functions on $\R$, and by  $C_c(\R)$ the set of compactly supported continuous functions on $\R$. Throughout this paper we shall also need the following definitions.
\begin{defn}[Weakly convergence of measures]\label{def_weakly_conv}
We say that a sequence $\{\nu_n\}\subseteq\mathcal{M}$ converge weakly towards $\nu\in\mathcal{M}$, which we shall denote by $\nu_n\rightharpoonup \nu$, if \[
 \langle \nu_n,\phi\rangle\to\langle \nu,\phi\rangle, \quad \hbox{for any }\, \phi\in C_c(\R).
\]
\end{defn}

\begin{rem}\label{weak_weakstar_conv}
Notice that we are adopting the standard Measure Theory's notation for the \emph{weak convergence} of a measure. Nevertheless, we recall that from a Functional Analysis point of view this convergence corresponds to the weak-* convergence on Banach spaces.
\end{rem}

\begin{defn}[Tightly and weak continuity of measure-valued functions]\label{def_c_ti} Let $I\subseteq \R$ be an interval.
\begin{enumerate}
\item We say that a function $f\in C_{ti}(I,\mathcal{M}_b)$ if for any $\phi\in C_b(\R)$ the map $t\mapsto\langle f(t)\phi\rangle$ is continuous on $I$.
\item We say that a function $f\in C_w(I,\mathcal{M})$ if for any $\phi\in C_c(\R)$ the map $t\mapsto\langle f(t)\phi\rangle$ is continuous in $I$.
\end{enumerate}
\end{defn}

\begin{defn}[Weak convergence in $C_{ti}(I)$] Let $I\subseteq \R$ be an interval. We say that a sequence $f_n\rightharpoonup f$ in $C_{ti}(I,\mathcal{M}_b)$ if for any $\phi \in C_b(\R)$ we have \[
\langle f_n(\cdot)\phi\rangle \to \langle f(\cdot)\phi\rangle \,\hbox{ in } C(I).
\]
\end{defn}

\subsection{Global well-posedness}

In the proofs of Theorem  \ref{MT5} we shall need to approximate non-smooth solutions of equation \eqref{nov_eq_2} by sequences of smooth solutions. In this regard, we shall need a global well-posedness result on a class of smooth solutions. In \cite{WuYi}, following the ideas of the seminal work of Constantin and Escher \cite{CoEs} on the Camassa-Holm equation, Wu and Yin proved the smooth global well-posedness for initial data with non-negative momentum density.

\begin{thm}[\cite{WuYi}]\label{GWP_smooth}
Let $u_0\in H^s$ for $s\geq 3$, with non-negative  momentum density $y_0$ belonging to $L^{1}(\R)$. Then, equation \eqref{novikov_eq} has a unique global strong solution \[
u\in C(\R,H^s(\R))\cap C^1(\R,H^{s-1}(\R)).
\]
Moreover, we have that $E(u)$ and $F(u)$ are two conservation laws. Additionally, denoting by $y(t):=u(t)-u_{xx}(t)$ the momentum density, we have that $y(t)$ and $u(t)$ are non-negative for all times $t\in\R$ and $\vert u_x(t,\cdot)\vert\leq u(t,\cdot)$ on $\R$.
\end{thm}
Unfortunately, since peakon profiles
do not belong\footnote{Actually, they do not belong to any $W^{1+\frac{1}{p},p}(\R)$ for any $p\in [1,+\infty)$. However, peakon profiles do belong to $W^{1,\infty}(\R)$, where $W^{1,\infty}(\R)$ denotes the space of Lipschitz functions.} to $H^{3/2}(\R)$, they do not enter into this framework either, and hence this theorem is not useful for our purposes. Nevertheless, by following the work of Constantin and Molinet \cite{CoMo}, in the same work Wu and Yin also proved a global well-posedness theorem for a class of functions containing peakons. This result shall be crucial in our analysis. However, we shall need a slightly improved version of this theorem, which we state below. 
\begin{thm}[\cite{WuYi}]\label{theorem_gwp}
Let $u_0\in H^1(\R)$ be a function satisfying $y_0:=(u_0-u_{0,xx})\in\mathcal{M}_b^+(\R)$.
Then, the following properties hold: \begin{enumerate}
\item[1.] \emph{\textbf{Uniqueness and global existence:}} There exists a global weak solution \[
u\in C(\R,H^1(\R))\cap C^1(\R,L^2(\R)),
\]
associated to the initial data $u(0)=u_0$ such that its momentum density \[
y(t,\cdot):=u(t,\cdot)-u_{xx}(t,\cdot)\in C_{ti}(\R,\mathcal{M}_b^+(\R)).
\]
Additionally $E(u)$ and $F(u)$ are conservation laws. Moreover, the solution is unique in the class \[
\{f\in C(\R,H^1(\R))\}\cap\{f-f_{xx}\in L^\infty(\R,\mathcal{M}_b^+)\}.
\]
\item[2.] \emph{\textbf{Continuity with respect to the initial data $H^1(\R)$:}} For any sequence $\{u_{0,n}\}_{n\in\N}$ bounded in $Y_+(\R)$  such that $ u_{0,n}\to u_0 \,\hbox{ in } H^1(\R)$,
the following holds: For any $T>0$, the family of solutions $\{u_{n}\}$ to equation \eqref{nov_eq_2} associated to $\{u_{0,n}\}$ satisfies \begin{align}\label{convergence_h1_ti}
u_n\to u \,\hbox{ in }\, C([-T,T],H^1(\R)) \quad \hbox{and} \quad y_{0,n}\rightharpoonup y \,\hbox{ in }\, C_{ti}([-T,T],\mathcal{M}).
\end{align}
\end{enumerate}
\end{thm}

\begin{proof}
We refer to \cite{Mo,Mo2}, Propositions $2.2$, for a proof of this theorem in both the Camassa-Holm and the $b$-family case. Notice that the same proof applies to the Novikov equation, provided Theorem \ref{GWP_smooth} and the fact that the first point of the statement was proven in \cite{WuYi}, except for the fact that $y\in C_{ti}(\R,\mathcal{M}_b^+)$, which can be proven in exactly the same fashion as in \cite{Mo}.
\end{proof}

\subsection{Local well-posedness}
To study the orbital stability of a train of peakons, since we are not assuming positivity of the momentum density we shall need a suitable well-posedness theorem. In this regard, in \cite{HiHo} the following local well-posedness for smooth solutions was derived. 
\begin{thm}[\cite{HiHo}]\label{lwp_smooth}
Let $u_0\in H^s$ with $s>\tfrac{3}{2}$. Then, there exists $T>0$ and a unique solution $ u\in C([0,T],H^s(\R))$ to equation \eqref{nov_eq_2} associated to $u_0$. Moreover, the data-to-solution map depends continuously on $u_0$.
\end{thm}
Nevertheless, as we discussed before, neither peakons nor peakon trains belong to this class of initial data. However, in \cite{Da} Danchin noticed that, in the Camassa-Holm case, the maximal existence time of a solution $u\in H^s(\R)$ for $s>\tfrac{3}{2}$ is bounded from below by a positive real number, which allowed him to obtain local weak solutions without the positivity assumption on the initial momentum density. The following theorem states the analogous result for the Novikov equation.
\begin{thm}\label{theorem_lwp}
Let $u_0\in H^1(\R)$ be a function satisfying $y_0:=(u_0-u_{0,xx})\in\mathcal{M}_b(\R)$.
Then, there exists a function $T=T(\Vert y_0\Vert_{\mathcal{M}})>0$ and a unique solution $u\in\mathcal{C}_{ti}([-T,T],Y(\R))$ to equation \eqref{nov_eq_2} associated to the initial data $u_0$. Moreover, the functionals $E(\cdot)$ and $F(\cdot)$ are conserved along the trajectory. Additionally, if $y_0\in\mathcal{M}_b$ defines a signed Radon measure, then the solution $u(t)$ is global in time. Furthermore, if $\{u_{0,n}\}$ is a sequence in $Y(\R)$ satisfying $u_{0,n}\to u_0$ in $Y(\R)$, then the corresponding sequence of solutions $\{u_n\}$ to the Novikov equation emanating from $u_{0,n}$ satisfy $u_n\to u$ in $C([-T,T],H^1(\R))$.
\end{thm}

\begin{proof} As we discussed before, the proof of local existence for weak solutions is mainly contained in \cite{Da} for the Camassa-Holm case. Nevertheless, for the sake of completeness, since it has not been shown for the Novikov equation we sketch the prove of it. The idea is to combine the smooth global well-posedness theorem \ref{lwp_smooth} with an apriori estimate of the total variation of the momentum density. In fact, let us fix $T<2 \Vert y_0\Vert_{\mathcal{M}}^{-2}$ and consider the family of smooth initial data $u_{n,0}:=\rho_n*u_0\in H^\infty$. Notice that by Young's inequality we have $\Vert y_{n,0}\Vert_{L^1}\leq \Vert y_0\Vert_{\mathcal{M}}$. Then, by Theorem \ref{lwp_smooth} there exists a unique solution \[
u\in C(\R,H^\infty(\R)) \ \, \hbox{ such that } \ \, u\big\vert_{t=0}=u_0.
\]
On the other hand, by \eqref{nov_eq_y} we know that $y_n:=u_n-u_{n,xx}$ solves
\[
y_{n,t}+u_n^2y_{n,x}+3u_nu_{n,x}y_{n}=0 \ \, \hbox{ and hence } \, \ \partial_t\vert y_{n}\vert +\partial_x(u_n^2\vert y_{n}\vert)=-\vert y_{n}\vert u_nu_{n,x}.
\]
Thus, after integration in the space variable and by using that $\Vert u_x\Vert_{L^\infty}\leq \tfrac{1}{2}\Vert y\Vert_{L^1}$ we deduce
\[
\dfrac{d}{dt}\Vert y_n(t)\Vert_{L^1}\leq \dfrac{1}{4}\Vert y_n(t)\Vert_{L^1}^3,
\]
which leads us to,
\[
\Vert y_n(t)\Vert_{L^1}\leq \left(\dfrac{2\Vert y_{0}\Vert_{\mathcal{M}}^2}{2-t\Vert y_{0}\Vert^2_{\mathcal{M}}}\right)^{1/2}.
\]
Finally, notice that the previous estimate give us an uniform bound on $\Vert u_{n,x}\Vert_{L^\infty}$ for $t\leq\min\{T_n,2\Vert y_0\Vert_{\mathcal{M}}^{-2}\}$, where $T_n$ denotes the maximal existence time of $u_n(t)$. Hence, due to the blow-up criteria, for all $n\in\N$ this uniform bound leads us to \[
\int_0^{T_{y_0}}\Vert u_{n,x}(t)\Vert_{L^\infty}dt<+\infty, \ \, \hbox{ and hence } \ \, T_n\geq 2\Vert y_0\Vert_{\mathcal{M}}^{-2}\geq T,
\]
where we are denoting by $T_{y_0}:=2\Vert y_0\Vert_{\mathcal{M}}^{-2}$. Therefore, we conclude the apriori estimate which leads us to the local existence. The remaining part of the proof follows from standard compactness arguments to justify the pass to the limit. We refer to \cite{Da} Theorem $5$ for a detailed proof of this last part (the interested reader may also consult to \cite{Mo4}, Theorem $4$ for a proof of this last part).
\end{proof}

\smallskip

\section{Orbital stability of peakons in the energy space}\label{sec_MT4}
In this section we show that under some slight improvements and modifications of the proof in \cite{LiLiQu} we are able to obtain Theorem \ref{MT4}. For the sake of simplicity we split the proof in several lemmas, which we shall state and prove in the next subsection. 

\subsection{General estimates}\label{general_estimates_peakon}
In this subsection we shall prove some general formulas and estimates holding for any function belonging to $H^1(\R)\cap W^{1,4}(\R)$, which are the minimum requirements for $E(u)$ and $F(u)$ to be well-defined.
\begin{lem}\label{lemma_three_one}
For any $v\in H^1(\R)$ and any $z\in\R$ we have \begin{align}\label{formula_energy}
E(v)-E(\varphi_c)=\Vert v-\varphi_c(\cdot-z)\Vert_{H^1}^2+4\sqrt{c}(v(z)-\sqrt{c}).
\end{align}
\end{lem}
\begin{rem}
Notice that the previous lemma ensures us that the minimum of the $H^1(\R)$-distance between $v$ and the set $\{\varphi_c(\cdot-\xi):\ \xi\in\R\}$ is reached exactly at any point $\xi\in\R$ where $v(\cdot)$ attains its maximum.
\end{rem}
\begin{proof}
The proof follows from direct computations. In fact, recalling that $\varphi_c''=\varphi-2\delta$, after integration by parts we obtain \begin{align*}
\Vert v-\varphi_c(\cdot-z)\Vert_{H^1}^2&=E(v)+E(\varphi_c)-2\int v(x)\varphi_c(\cdot-z)dx-2\int v_x(x)\varphi_c'(\cdot-z)dx
\\ & =E(v)+E(\varphi_c)-4\sqrt{c}v(z)=E(v)-E(\varphi_c)-4\sqrt{c}(v(z)-\sqrt{c}).
\end{align*}
The proof is complete.
\end{proof}

\begin{lem}\label{lemma_three_two}
Let $v\in H^1(\R)\cap W^{1,4}(\R)$ and let $\xi\in\R$ be any point where $v(\cdot)$ attains it maximum, that is, $v(\xi)=\max_\R v(x)$. Then, denoting this quantity by $M:=v(\xi)$ we have, \[
F(v)\leq \dfrac{4}{3} M^2E(v)-\dfrac{4}{3}M^4.
\]
\end{lem}

\begin{proof}
First of all let us start by introducing some notation. From now on we denote by $g$ and $h$ the functions given by
\[
g(x):=\begin{cases}
v(x)-v_x(x), & x<\xi,
\\ v(x)+v_x(x), & x>\xi,
\end{cases} \qquad 
h(x):=\begin{cases}
v^2-\tfrac{2}{3}vv_x-\tfrac{1}{3}v_x^2, & x<\xi,
\\ v^2+\tfrac{2}{3}vv_x-\tfrac{1}{3}v_x^2, & x>\xi.
\end{cases}
\]
Then, on the one-hand, by direct computations we have \begin{align*}
\int h(x)g^2(x)dx&=\int_{-\infty}^{\xi} \big(v^2-\tfrac{2}{3}vv_x-\tfrac{1}{3}v_x^2\big)(v-v_x)^2dx
\\ & \quad +\int_\xi^{+\infty}\big(v^2+\tfrac{2}{3}vv_x-\tfrac{1}{3}v_x^2\big)(v+v_x)^2=:\mathrm{I}+\mathrm{II}.
\end{align*}
Thus, rearranging and simplifying terms we obtain \begin{align*}
\mathrm{I}&=\int_{-\infty}^{x_i}\big(v^4+2v^2v_x^2-\tfrac{1}{3}v_x^4-\tfrac{8}{3}v^3v_x\big)dx=\int_{-\infty}^{x_i}\big(v^4+2v^2v_x^2-\tfrac{1}{3}v_x^4\big)dx-\dfrac{2}{3}M^4
\end{align*}
Similarly, rearranging and simplifying terms we get \begin{align*}
\mathrm{II}=\int_\xi^{+\infty}\big(v^4+2v^2v_x^2-\tfrac{1}{3}v_x^4\big)dx-\dfrac{2}{3}M^4.
\end{align*}
Plugging the last two formulas together we obtain \begin{align}\label{F_first_id}
\int h(x)g^2(x)dx=F(v)-\dfrac{4}{3}M^4.
\end{align}
On the other hand, by using $a^2+b^2\geq 2ab$ we have \[
v^2\pm\dfrac{2}{3}vv_x-\dfrac{1}{3}v_x^2\leq \dfrac{4}{3}v^2.
\]
Therefore, by using the latter inequality we deduce $h(x)\leq \tfrac{4}{3}v^2$ and hence \begin{align}\label{F_second_id}
\int h(x)g^2(x)dx&\leq \dfrac{4}{3}\int v^2(x)g^2(x)=\dfrac{4}{3}\int_{-\infty}^\xi v^2(v-v_x)^2dx+\dfrac{4}{3}\int_\xi^{+\infty}v^2(v+v_x)^2\nonumber
\\ & \leq \dfrac{4}{3}M^2\int_{-\infty}^\xi \big(v^2+v_x^2-2vv_x\big)dx+\dfrac{4}{3}M^2\int_\xi^{+\infty}\big(v^2+v_x^2-2vv_x\big)dx\nonumber
\\ & = \dfrac{4}{3}M^2E(v)-\dfrac{8}{3}M^4.
\end{align}
Gathering \eqref{F_first_id} with \eqref{F_second_id} we obtain the desired result.
\end{proof}

The following lemma gives us an estimate of the distances between the evaluations of $E$ and $F$ at $u$ and $\varphi_c$ in terms of the distance of $u_0$ and $\varphi_c$ in $H^1\cap W^{1,4}$.
\begin{lem}\label{E_F_differences}
Let $v\in H^1(\R)\cap W^{1,4}(\R)$ any function satisfying $\Vert v-\varphi_c\Vert_{H^1}+\Vert v_x+\varphi_c'\Vert_{L^4}<\epsilon$ for some $\min\{1,c\}\gg\epsilon>0$. Then, \begin{align*}
\vert E(v)-E(\varphi_c)\vert\leq4\sqrt{c}\epsilon \quad \hbox{and}\quad \vert F(v)-F(\varphi_c)\vert \leq 120c^{3/2}\epsilon.
\end{align*}
\end{lem}

\begin{proof}
Let us start estimating the difference of the energies. First of all notice we recall that  \begin{align}\label{recall_E_F}
E(\varphi_c)=2c \quad \hbox{and}\quad  F(\varphi_c)=\tfrac{4}{3}c^2.
\end{align}
Hence, by using the hypothesis, triangular inequality and the previous identities we obtain \[
\Vert v\Vert_{H^1}\leq \Vert \varphi_c\Vert_{H^1}+\epsilon\leq 2\sqrt{c} \quad \hbox{and}\quad \Vert v_x\Vert_{L^4}\leq \Vert \varphi_c'\Vert_{L^4}+\epsilon\leq \sqrt{c}.
\]
Then, by using the reverse triangular inequality we get \begin{align*}
\vert E(v)-E(\varphi_c)\vert &\leq \big(\Vert v\Vert_{H^1}+\Vert \varphi_c\Vert_{H^1}\big)\Big\vert\Vert v\Vert_{H^1}-\Vert \varphi_c\Vert_{H^1}\Big\vert
\\ & \leq 4\sqrt{c}\Vert v-\varphi_c\Vert_{H^1}\leq 4\sqrt{c}\epsilon.
\end{align*}
For the second estimate let us start by rearranging the integral terms involved in $F(\cdot)$. In fact, it easy to see that
\begin{align*}
\vert F(v)-F(\varphi)\vert&=\left\vert\int \left(v^4+2v^2v_x^2-\tfrac{1}{3}v_x^4\right)-\int\left(\varphi_c^4+2\varphi_c\varphi_c'^2-\tfrac{1}{3}\varphi_c'^4\right)\right\vert
\\ & \leq \left\vert\int \big(v^2+v_x^2\big)^2-\big(\varphi_c^2+\varphi_c'^2\big)^2\right\vert+\dfrac{4}{3}\left\vert\int \big(v_x^4-\varphi_c'^4\big)\right\vert=:\mathrm{I}+\mathrm{II}.
\end{align*}
Now notice that, on the one hand, by using H\"older's and triangular inequalities, together with Sobolev's embedding $H^1\hookrightarrow L^4$ we obtain 
\begin{align*}
\mathrm{I}&=\left\vert\int \big(v^2+v_x^2+\varphi_c^2+\varphi_c'^2\big)\big(v^2+v_x^2-\varphi_c^2-\varphi_c'^2\big)\right\vert
\\ &=\left\vert\int \big(v^2+v_x^2+\varphi_c^2+\varphi_c'^2\big)\Big((v^2-\varphi_c^2)+(v_x^2-\varphi_c'^2)\Big)\right\vert
\\ & \leq 2\big(\Vert v\Vert_{W^{1,4}}^2+\Vert \varphi_c\Vert_{W^{1,4}}^2\big)\big(\Vert v+\varphi_c\Vert_{L^4}\Vert v-\varphi_c\Vert_{L^4}+\Vert v_x+\varphi_c'\Vert_{L^4}\Vert v_x-\varphi_c'\Vert_{L^4}\big)
\\ & \leq 100c^{3/2}\epsilon.
\end{align*}
On the other hand, by using H\"older's and triangular inequality again we get
\begin{align*}
\mathrm{II}&=\dfrac{4}{3}\left\vert\int (v_x^4-\varphi_c'^4)\right\vert=\dfrac{4}{3}\left\vert\int (v_x^2+\varphi_c'^2)(v_x^2-\varphi_c'^2)\right\vert
\\ & = \dfrac{4}{3}\left\vert\int (v_x^2+\varphi_c'^2)(v_x+\varphi_c')(v_x-\varphi_c')\right\vert
\\ & \leq 2\Vert v_x +\varphi_c'\Vert_{L^4}^3\Vert v_x-\varphi_c'\Vert_{L^4}\leq 20c^{3/2}\epsilon.
\end{align*}
Gathering both estimates we obtain the desired result. The proof is complete.
\end{proof}
We finish this section by estimating the remaining term in formula \eqref{formula_energy} when choosing  $\xi\in\R$ to be the natural candidate to study the orbital stability of $\varphi_c$.
\begin{lem}\label{lemma_three_four}
Let $v\in H^1(\R)\cap W^{1,4}(\R)$ arbitrary and let us denote by $M:=\max_\R v(\cdot)$. Assume that for some $0<\epsilon\ll\min\{1,c\}$ the following estimates are satisfied \begin{align}\label{hyp_lemma_three_four}
\vert E(v)-E(\varphi_c)\vert\leq4\sqrt{c}\epsilon \quad \hbox{and}\quad \vert F(v)-F(\varphi_c)\vert \leq 120c^{3/2}\epsilon.
\end{align}
Then, the following estimate holds \[
\vert M-\sqrt{c}\vert\leq 10c^{3/4}\epsilon^{1/2}.
\]
\end{lem}

\begin{proof}
First of all we recall that, by Lemma \ref{E_F_differences} we have \[
F(v)-\dfrac{4}{3} M^2E(v)+\dfrac{4}{3}M^4\leq0 \ \,\hbox{ and hence }\ \, M^4-M^2E(v)+\dfrac{3}{4}F(v)\leq0.
\]
Now, let us introduce the fourth-order polynomials $P(q)$ and $\widetilde{P}(q)$ which are given by \[
P(q):=q^4-E(v)q^2+\dfrac{3}{4}F(v) \quad \hbox{and}\quad \widetilde{P}(q):=q^4- E(\varphi_c)q^2+\dfrac{3}{4}F(\varphi_c).
\]
By evaluating the latter polynomial in $M$ and rearranging terms we have
\[
\widetilde{P}(M)=P(M)+\big(E(v)-E(\varphi_c)\big)M^2-\dfrac{3}{4}\big(F(v)-F(\varphi_c)\big).
\]
At this point it is worth noticing that, due to both identities in \eqref{recall_E_F}, we can rewrite $\widetilde{P}(q)$ as: \[
\widetilde{P}(q)= \left(q+\sqrt{c}\right)^2\left(q-\sqrt{c}\right)^2.
\]
Noticing that $M\geq 0$ and due to the fact that (by Lemma \ref{E_F_differences}) $P(M)\leq 0$, we deduce that \begin{align}\label{final_estimate}
c(M-\sqrt{c})^2\leq (M+\sqrt{c})^2(M-\sqrt{c})^2\leq \big(E(v)-E(\varphi_c)\big)M^2-\dfrac{3}{4}\big(F(v)-F(\varphi_c)\big).
\end{align}
Therefore, by using both hypotheses in \eqref{hyp_lemma_three_four} and due to the fact that $M\leq 2\sqrt{c}$, we conclude \[
\sqrt{c}\vert M-\sqrt{c}\vert\leq 10c^{3/4}\epsilon^{1/2}.
\]
The proof is complete.
\end{proof}

\subsection{Proof of Theorem \ref{MT4}}
With all the previous lemmas we are able to conclude the proof of Theorem \ref{MT4}. First of all, we recall that since $E(\cdot)$ and $F(\cdot)$ are conserved along the trajectory, for any $t\in [0,T)$ we have \begin{align}\label{cons_E_F}
E(u(t))=E(u_0) \quad \hbox{and}\quad F(u(t))=F(u_0).
\end{align}
Now, notice that by applying Lemma \ref{lemma_three_one}, for any time $t\in[0,T)$ we have 
\begin{align}\label{norm_equality}
\Vert u(t)-\varphi_c(\cdot-\xi(t))\Vert_{H^1}^2=E(u_0)-E(\varphi_c)-4\sqrt{c}\Big(v(t,\xi(t))-\sqrt{c}\Big),
\end{align}
where $\xi(t)$ denotes any space-point where $v(t,\cdot)$ attains its maximum, i.e. $M(t)=v(t,\xi(t))=\max_\R v(t,\cdot)$. On the other hand, by using Lemma \ref{E_F_differences} with $u_0$ and $\epsilon=\varepsilon^4$ together with the conservation laws \eqref{cons_E_F} we deduce that $u(t)$ satisfies the hypothesis of Lemma \ref{lemma_three_four} for all times $t\in[0,T)$. Finally, notice that the right-hand side of estimate \eqref{final_estimate} only depends conserved quantities, and hence we obtain \[
\vert M(t)-\sqrt{c}\vert\leq 10c^{1/4}\varepsilon^2.
\]
Therefore, plugging the latter inequality into \eqref{norm_equality} we conclude \[
\Vert u(t)-\varphi_c(\cdot-\xi(t))\Vert_{H^1}^2\leq 4\sqrt{c}\varepsilon^4+40c^{3/4}\varepsilon^2.
\]
The proof is complete.  \qed

\medskip

\section{Orbital stability of a train of peakons}\label{sec_MT2}

The proof of the orbital stability for peakon trains follows similar ideas to those shown for the single peakon. Thus, as in the previous section, for the sake of simplicity we shall split the proof of Theorem \ref{MT2} in several lemmas which we shall state and prove in the next subsection. 

\medskip

On the other hand, in order to make the computations more understandable we shall need to introduce some extra notation for the sum of peakons. From now on, for any choice of speeds $(c_1,...,c_n)\in \R_+^n$ and any vector $\vec{z}:=(z_1,...,z_n)\in\R^n$ we shall denote by $R_{\vec{z}}$ the sum of $n$ peakons with speeds $c_1,...,c_n$ centered at $z_1,...,z_n$ respectively, that is
\[
R_{\vec{z}}(x):=\sum_{i=1}^n\varphi_{c_i}(x-z_i)=\sum_{i=1}^n\sqrt{c_i}e^{-\vert x-z_i\vert }.
\]
Now, before getting into the proof we shall need a modulation lemma in order to ensure that no strong interactions between different peakons happen.

\subsection{Modulation around multipeakons}\label{modulation_section}

Let $\alpha$ and $L$ any pair of positive real numbers. We consider the neighborhood of radius $\alpha$ around the sum of all well-ordered $n$ peakons with speeds $c_1,...,c_n$ separated by at least $L$, i.e,  
\begin{align}\label{tubular}
\mathcal{U}(\alpha,L):=\left\{u\in H^1(\R):\ \inf_{x_j-x_{j-1}>L}\left\Vert u-\sum_{i=1}^n\varphi_{c_i}(\cdot-x_i)\right\Vert_{H^1}<\alpha\right\}.
\end{align}
By a bootstrapping argument and due to the continuity of the map $t\mapsto u(t)$ from $[0,T)$ into $H^1(\R)$, in order to conclude Theorem \ref{MT2} it is sufficient to prove that the following holds: There exist $C>0$, $\varepsilon^\star>0$ and $L_0>0$ such that for all $L\geq L_0$ and $\varepsilon\in(0,\varepsilon^\star)$, if a solution $u\in C([0,T),H^1(\R))\cap L^\infty([0,T),W^{1,4}(\R))$ to the Novikov equation \eqref{nov_eq_2} satisfying the hypothesis of Theorem \ref{MT2} is such that there exists $t^*\in (0,T)$ with the property:
\begin{align}\label{neigh_assumption}
u(t)\in\mathcal{U}\left(C(\varepsilon+L^{-1/8}),\tfrac{1}{2}L\right),\ \, \hbox{ for all } \ t\in[0,t^*],
\end{align}
then, at $t=t^*$ we have
\begin{align}\label{neigh_conclusion}
u(t^*)\in\mathcal{U}\left(\tfrac{C}{2}(\varepsilon+L^{-1/8}),\tfrac{2}{3}L\right).
\end{align}
Therefore, in the rest of this section we shall assume that \eqref{neigh_assumption} holds for some $\varepsilon\in(0,\varepsilon^ \star)$ and some $L>L_0$ and we shall prove that under these hypothesis we have \eqref{neigh_conclusion}.

\medskip

The next lemma ensure us that the different bumps of $u(t)$ that are individually close to a peakon get away of each other as time evolves. This shall be crucial in the sequel.
\begin{lem}\label{mod_lemma}
Let $u\in C([0,T),H^1(\R))\cap L^\infty([0,T),W^{1,4}(\R))$ be a solution to the Novikov equation \eqref{nov_eq_2} satisfying the assumptions of Theorem \ref{MT2}. There exists $\alpha_0>0$ small enough and $L_0>0$ sufficiently large such that for any $0<\alpha<\alpha_0$ the following holds: If for some $t^*\in(0,T)$ the solution $u(t)$ satisfies \begin{align}\label{tubular_neigh}
u(t)\in \mathcal{U}(\alpha,\tfrac{1}{2}L)  \ \, \hbox{ for all } \ \,t\in[0,t^*],
\end{align}
then there exist $C^1$ functions $\widetilde{x}_1,...,\widetilde{x}_n:[0,t^*]\to\R$ such that \[
u(t,x)=\sum_{i=1}^n\varphi_{c_i}\big(x-\widetilde{x}_i(t)\big)+v(t,x),
\] 
where $\{\widetilde{x}_i\}_{i=1}^n$ are chosen in such a way that for all $t\in[0,t^*]$ the following orthogonality conditions hold \begin{align}\label{mod_orthogonality}
\int_\R v(t,x)\big(\rho_{n_0}*\varphi_{c_i}'\big)(\cdot-\widetilde{x}_i(t))dx=0, \ \, \hbox{ for all } \, i=1,...,n,
\end{align}
where $\rho_n$ is defined in \eqref{def_rho} and $n_0\in\N$ satisfies: \begin{align}\label{orth_cond_def}
\hbox{For all }\, -\tfrac{1}{2}\leq y\leq \tfrac{1}{2}, \quad \int \varphi(\cdot-y)\big(\rho_{n_0}*\varphi'\big)=0 \ \iff \ y=0.
\end{align}
Moreover, with this election of shifts we have that there exists $C>0$ such that for all $t\in[0,t^*]$ we have: \begin{align}\label{mod_bound}
\left\Vert u(t)-\sum_{i=1}^n\varphi_{c_i}\big(\cdot-\widetilde{x}_i(t)\big)\right\Vert_{H^1}\leq C\alpha^{1/2}.
\end{align}
Furthermore, for all $i=1,...n$ ($n-1$ respectively) and all $t\in[0,t^*]$ the following estimates hold:
\begin{align}\label{mod_parameter_bound}
\left\vert \dot{\widetilde{x}}_i(t)-c_i\right\vert \leq C\alpha^{1/2} \quad \hbox{and }\quad  \widetilde{x}_i(t)-\widetilde{x}_{i-1}(t)\geq \tfrac{3}{4}L+\tfrac{1}{2}(c_i-c_{i-1})t.
\end{align}
Additionally, by setting the family of time-dependent intervals $\mathcal{J}_i(t):=[y_i(t),y_{i+1}(t)]$, where \begin{align}\label{def_y_i_intervals}
y_1=-\infty, \quad y_{n+1}=+\infty, \ \, \hbox{ and }\ \, y_i(t)=\tfrac{1}{2}\big(\widetilde{x}_{i-1}(t)+\widetilde{x}_i(t)\big),
\end{align}
then, for all $t\in[0,t^*]$ there exists $x_i(t)\in\mathcal{J}_i(t)$ for $i=1,...,n$, such that \begin{align*}
u\big(t,x_i(t)\big)=\max_{x\in\mathcal{J}_i(t)}u(t,\cdot) \quad \hbox{and}\quad \big\vert x_i(t)-\widetilde{x}_i(t)\big\vert\leq \tfrac{1}{12}L.
\end{align*}
\end{lem}

\begin{proof}
See the appendix, Section \ref{mod_appendix}.
\end{proof}

\subsection{Almost monotonicity property}

By using the previous lemma we shall define the modified energy functionals measuring the energy at the right of each bump of $u(t)$. In fact, from Lemma \ref{mod_lemma} we deduce the existence of $C^1$ functions $\widetilde{x}_1,...,\widetilde{x}_n$ satisfying \eqref{mod_orthogonality}-\eqref{mod_parameter_bound}. From now on we shall denote by $\Psi_{i,K}$ the family of weight functions given by \begin{align}\label{def_Psi_i_K}
\Psi_{i,K}=\Psi\left(\dfrac{x-y_i(t)}{K}\right) \quad \hbox{where}\quad \Psi(x):=\dfrac{2}{\pi}\arctan\left(e^x\right),
\end{align}
where the family $\{y_i\}_{i=1}^n$ is defined in \eqref{def_y_i_intervals}. Now, for each $i=1,...,n$ and $K>1$, we define the modified energy functional
\begin{align*}
\mathcal{I}_{i,K}(t)=\mathcal{I}_{i,K}\big(u(t)\big):=\int \big( u^2+u_x^2\big)(t,x)\Psi_{i,K}(t,x)dx.
\end{align*}
As we discussed before, the idea of defining these functionals is to be able to measure the energy of $u(t)$ at the right of each bump. In particular, for all times $t\in[0,T)$ we have \[
\mathcal{I}_{i,K}(t)\geq \dfrac{1}{2}\Vert u(t)\Vert_{H^1(y_i(t),+\infty)}.
\]
Finally, let us fix $\sigma_0:=\tfrac{1}{4}\min\{c_1,c_2-c_1,...,c_n-c_{n-1}\}$. The following lemma give us the almost monotonicity property of the energy at the right.
\begin{lem}\label{AM_orb_train}
Let $u\in C([0,T),H^1(\R))\cap L^\infty([0,T),W^{1,4}(\R))$ be a solution to the Novikov equation \eqref{nov_eq_2} satisfying \eqref{mod_bound} on $[0,t^*]$. Then, there exists $\alpha_0>0$ small enough and $L_0>0$, only depending on $c_1$, such that if $\alpha<\alpha_0$ and $L\geq L_0$ then, for any $4\leq K\lesssim L^{1/2}$ the following holds
\begin{align}\label{AM_energy_orbital}
\mathcal{I}_{i,K}(t)-\mathcal{I}_{i,K}(0)\leq O\left(e^{-L/8K}\right), \ \, \hbox{ for all } \, i=2,...,n, \, \hbox{ and all } \,t\in[0,t^*].
\end{align}
\end{lem}

\begin{proof}
See the appendix, Section \ref{AM_orb_train_appendix}.
\end{proof}

\subsection{General estimates}
In this subsection we shall prove some general formulas and estimates holding for any function belonging to $H^1(\R)\cap W^{1,4}(\R)$. All of these formulas and estimates are just the localized versions of the ones in Section \ref{general_estimates_peakon}. In this regard we shall need the following definitions. Let us consider the family of functions $\Phi_i(t,x)$ given by \[
\Phi_1(t,x):=1-\Psi_{2,K}(t,x), \quad \Phi_n(t,x):=\Psi_{n,K}(t,x) \ \, \hbox{ and } \ \,  \Phi_i(t,x):=\big(\Psi_{i,K}-\Psi_{i+1,K}\big)(t,x).
\]
It is important to point out that this family of functions satisfies $\sum_{i=1}^n\Phi_{i,K}\equiv1$. On the other hand, notice that for $L,K>0$ large enough and every $i\neq j$ we have \begin{align*}
\vert 1-\Phi_{i,K}\vert\leq 4e^{-L/4K} \ \hbox{ and } \ \vert \Phi_{j,K}\vert \leq 4e^{-L/4K} \ \, \hbox{ on } \ \, \left[\widetilde{x}_i-\tfrac{L}{4},\widetilde{x}_i+\tfrac{L}{4}\right]
\end{align*}
In the sequel we shall need localized versions of the conservation laws in \eqref{cons_e}. In this regard, from now on we denote by $E_i$ and $F_i$ the localized functionals given by
\begin{align*}
E_i(u)&:=\int \big(u^2+u_x^2\big)(t,x)\Phi_i(t,x)dx,
\\ F_i(u)&:=\int \left(u^4+2u^2u_x^2-\dfrac{1}{3}u_x^4\right)(t,x)\Phi_i(t,x)dx.
\end{align*}

The next lemma give us a global identity which shall be crucial in the sequel.
\begin{lem}\label{lemm_four_three}
For any vector $\vec{z}\in\R^n$ satisfying $z_i-z_{i-1}>\tfrac{1}{2}L$ and any function $v\in H^1(\R)$ we have
\begin{align}\label{formula_train_energy}
E(v)-\sum_{i=1}^nE(\varphi_{c_i})=\Vert v-R_{\vec{z}}\Vert_{H^1}^2+4\sum_{i=1}^n\sqrt{c_i}\big(v(z_i)-\sqrt{c_i}\big)+O\left(e^{-L/4}\right).
\end{align}
\end{lem}

\begin{proof}
The proof follows from direct computations. In fact, recalling that $\varphi_{c_i}'(x)=-\sgn(x)\varphi_{c_i}(x)$ and by integrating by parts we obtain
\begin{align*}
E\big(v-R_{\vec{z}}\big)&=E(v)+E(R_{\vec{z}})-2\sum_{i=1}^n\int v(x)\varphi_{c_i}(\cdot-z_i)+v_x(x)\varphi'_{c_i}(\cdot-z_i)dx
\\&=E(v)+E(R_{\vec{z}})-2\sum_{i=1}^n\int v(x)\varphi_{c_i}(\cdot-z_i)dx
\\ & \quad -2\sum_{i=1}^n\int_{-\infty}^{z_i} v_x(x)\varphi_{c_i}(\cdot-z_i)dx+2\sum_{i=1}^n\int_{z_i}^{+\infty}v_x\varphi_{c_i}(\cdot-z_i)dx
\\ & = E(v)+E(R_{\vec{z}})-4\sum_{i=1}^n\sqrt{c_i}v(z_i)
\end{align*}
On the other hand, notice that since $z_i-z_{i-1}\geq \tfrac{1}{2}L$ we have \[
E(R_{\vec{z}})=\sum_{i=1}^n E(\varphi_{c_i})+O\left(e^{-L/4}\right)=2\sum_{i=1}^nc_i+O\left(e^{-L/4}\right).
\]
Gathering the last two formulas we obtain the desired result. Notice that the implicit constant involved in $O\left(e^{-L/4}\right)$ only depends on $c_1,...,c_n$. The proof is complete.
\end{proof}

\textbf{Important:} From now on we fix $K=\tfrac{1}{8}L^{1/2}$. 

\medskip

The following lemma is the localized version of Lemma \ref{lemma_three_two} and shall be crucial in the sequel.

\begin{lem}\label{lemm_four_four}
Let $u(t,x)$ be the solution of the Novikov equation \eqref{nov_eq_2} associated to $u_0\in H^1(\R)\cap W^{1,4}(\R)$, satisfying the hypothesis of Lemma \ref{mod_lemma} on $[0,t^*]$ with $\alpha$ given by \eqref{neigh_assumption}. Then, for all $t\in[0,t^*]$ the following inequality holds: \begin{align}\label{F_E_exp}
F_i(u)\leq \dfrac{4}{3}M_i^2E_i(u)-\dfrac{4}{3}M_i^4+\Vert u_0\Vert_{H^1}^4O\left(L^{-1/2}\right),
\end{align}
where $M_i$ denotes the local maximum $M_i:=\max\{u(t,x):\ x\in\mathcal{J}_i(t)\}$.
\end{lem}
\begin{proof}
First of all let us start by introducing some notation. For each $i=1,...,n$ we define $g_i$ and $h_i$ the functions given by
\[
g_i(t,x):=\begin{cases}
u-u_x, & x<x_i(t),
\\ u+u_x, & x>x_i(t),
\end{cases} \qquad 
h_i(t,x):=\begin{cases}
u^2-\tfrac{2}{3}uu_x-\tfrac{1}{3}u_x^2, & x<x_i(t),
\\ u^2+\tfrac{2}{3}uu_x-\tfrac{1}{3}u_x^2, & x>x_i(t).
\end{cases}
\]
Then, on the one-hand, by direct computations we have
\begin{align*}
\int h_i(t,x)g_i^2(t,x)\Phi_i(t,x)dx&=\int_{-\infty}^{x_i}\left(u^2-\tfrac{2}{3}uu_x-\tfrac{1}{3}u_x^2\right)(u-u_x)^2\Phi_i(t,x)
\\ & \quad +\int_{x_i}^{+\infty}\left(u^2+\tfrac{2}{3}uu_x-\tfrac{1}{3}u_x^2\right)(u+u_x)^2\Phi_i(t,x)=:\mathrm{I}+\mathrm{II},
\end{align*}
Now, by integration by parts we obtain
\begin{align*}
\mathrm{I}&=\int_{-\infty}^{x_i} \big(u^4+2u^2u_x^2-\tfrac{1}{3}u_x^4-\tfrac{8}{3}u^3u_x\big)\Phi_i(t,x)dx
\\ & =\int_{-\infty}^{x_i} \big(u^4+2u^2u_x^2-\tfrac{1}{3}u_x^4\big)\Phi_i(t,x)dx+\dfrac{2}{3}\int_{-\infty}^{x_i}u^4\Phi_i'(t,x)dx-\dfrac{2}{3}M_i^4\Phi_i(t,x_i).
\end{align*}
Similarly, by integration by parts again we get
\begin{align*}
\mathrm{II}=\int_{x_i}^{+\infty} \big(u^4+2u^2u_x^2-\tfrac{1}{3}u_x^4\big)\Phi_i(t,x)dx-\dfrac{2}{3}\int_{x_i}^{+\infty}u^4\Phi_i'(t,x)dx-\dfrac{2}{3}M_i^4\Phi_i(t,x_i).
\end{align*}
Plugging the last two formulas together we deduce
\begin{align}\label{formula_h_g}
\int h_i(t,x)g_i^2(t,x)\Phi_i(t,x)dx=F_i(u)-\dfrac{4}{3}M_i^4\Phi_i(x_i)+\dfrac{2}{3}\int_{-\infty}^{x_i}u^4\Phi_i'dx-\dfrac{2}{3}\int_{x_i}^{+\infty}u^4\Phi_i'dx.
\end{align}
On the other hand, notice that by using $a^2+b^2\geq 2ab$ we have \[
u(t,x)^2\pm \dfrac{2}{3}u(t,x)u_x(t,x)-\dfrac{1}{3}u_x^2(t,x)\leq \dfrac{4}{3}u^2(t,x)
\]
Thus, by using the latter inequality deduce $h_i(x)\leq \tfrac{4}{3}u^2$ and hence, by using \eqref{neigh_assumption} we get
\begin{align*}
\int h_ig_i^2\Phi_i&\leq \dfrac{4}{3}\int u^2g_i^2\Phi_i \leq  \dfrac{4}{3}M_i^2\int g_i^2\Phi_i+O\left(e^{-L^{1/2}}\right)
\\ & =\dfrac{4}{3}M_i^2E_i(u)-\dfrac{8}{3}M_i^4\Phi_i(x_i)+\dfrac{4}{3}M_i^2\int_{-\infty}^{x_i}u^2\Phi_i'-\dfrac{4}{3}M_i^2\int_{x_i}^{+\infty}u^2\Phi_i'+O\left(e^{-L^{1/2}}\right).
\end{align*}
Now it is important to notice that, since $\vert x_i-z_i\vert\leq\tfrac{1}{20}L$, we deduce that $\Phi_i(x_i)=1+O\big(e^{-L^{1/2}}\big)$. Therefore, gathering the latter inequality with \eqref{formula_h_g} we obtain 
\begin{align}\label{final_est_F_i_lemma}
F_i(u)&\leq \dfrac{4}{3}M_i^2E_i(u)-\dfrac{4}{3}M_i^4+\dfrac{4}{3}M_i^2\int_{-\infty}^{x_i}u^2\Phi_i'-\dfrac{4}{3}M_i^2\int_{x_i}^{+\infty}u^2\Phi_i'\nonumber
\\ & \quad -\dfrac{2}{3}\int _{-\infty}^{x_i}u^4\Phi_i'+\dfrac{2}{3}\int_{x_i}^{+\infty}u^4\Phi_i'+O\big(e^{-L^{1/2}}\big).
\end{align}
Finally, we recall that since $K=\tfrac{1}{8}L^{1/2}$ we have $\vert \Phi_i'\vert\leq \tfrac{C}{K}=O(L^{-1/2})$. Therefore, we conclude the proof of \eqref{F_E_exp} by plugging the latter estimate for $\Phi_i'$ on $\R$ into \eqref{final_est_F_i_lemma}. 
\end{proof}
As a consequence of the previous lemmas we obtain the following corollary.
\begin{lem}
Under the hypothesis of Lemma \ref{mod_lemma} and by considering $x_1(t),...,x_n(t)$ the functions constructed in such lemma, the following holds: For all $t\in[0,t^*]$ we have \begin{align}\label{lemma_tubular}
\left\Vert u(t)-\sum_{i=1}^n\varphi_{c_i}(\cdot-x_i(t))\right\Vert_{H^1}\leq O(\sqrt{\alpha})+\left(e^{-L/8}\right).
\end{align}
\end{lem}

\begin{proof}
In fact, recalling that by hypothesis we have $u(t)\in\mathcal{U}(\alpha,\tfrac{1}{2}L)$ for all $t\in[0,t^*]$, on account of Lemma \ref{mod_lemma} there exists $\widetilde{x}_1(t),...,\widetilde{x}_n(t)$ such that $\widetilde{x}_i(t)\in\mathcal{J}_i(t)$ and \[
\left\Vert u(t)-\sum_{i=1}^n\varphi_{c_i}(\cdot-\widetilde{x}_i(t))\right\Vert_{H^1}=O(\sqrt{\alpha}).
\]
Finally, recalling that $u\big(t,x_i(t)\big)=\max_{\mathcal{J}_i(t)}u(t,\cdot)$, by applying Lemma \ref{lemm_four_three} we conclude \begin{align*}
\left\Vert u(t)-\sum_{i=1}^n\varphi_{c_i}(\cdot-x_i(t))\right\Vert_{H^1}^2&=\left\Vert u(t)-\sum_{i=1}^n\varphi_{c_i}(\cdot-\widetilde{x}_i(t))\right\Vert_{H^1}^2
\\ & \quad -4\sum_{i=1}^n\sqrt{c_i}\big(u(t,x_i(t))-u(t,\widetilde{x}_i(t))\big)+O\left(e^{-L/4}\right)
\\ & \leq O(\alpha)+O\left(e^{-L/4}\right),
\end{align*}
which lead us to the desired result.
\end{proof}

Finally, the next lemma gives us a more accurate estimate of the last non-negligible term in formula \eqref{formula_train_energy} once we choose $\{z_i\}_{i=1}^n$ to be the natural candidate to study the orbital stability of a train of peakons, that is, the family of local maxima $\{x_i\}_{i=1}^n$ given by Lemma \ref{mod_lemma}.
\begin{lem}\label{lem_four_six}
Let $u(t,x)$ be the solution of the Novikov equation \eqref{nov_eq_2} associated to $u_0\in H^1(\R)\cap W^{1,4}(\R)$, satisfying the hypothesis of Lemma \ref{mod_lemma} on $[0,t^*]$ with $\alpha$ given by \eqref{neigh_assumption}. Then, for all $t\in[0,t^*]$ we have \begin{align}
\sum_{i=1}^n\sqrt{c_i}\left\vert M_i-\sqrt{c_i}\right\vert\leq O(\varepsilon^2)+O\left(L^{-1/4}\right).
\end{align}
\end{lem}

\begin{proof}
In fact, first of all we recall that for every $i=1,...,n$ the associated peakon profile satisfies \[
E(\varphi_{c_i})=2c_i \quad \hbox{and}\quad F(\varphi_{c_i})=\tfrac{4}{3}c_i^2.
\]
Hence, due to the fact that $M_i$ is positive and by using Lemma \ref{lemm_four_four}, we have
\begin{align*}
c_i\big(M_i-\sqrt{c_i}\big)^2&\leq \big(M_i+\sqrt{c_i}\big)^2\big(M_i-\sqrt{c_i}\big)^2=M_i^4-M_i^2E(\varphi_{c_i})+\dfrac{3}{4}F(\varphi_{c_i})
\\ & \leq \Big(M_i^2E_i(u(t))-M_i^2E(\varphi_{c_i})\Big)-\dfrac{3}{4}\Big(F_i(u(t))-F(\varphi_{c_i})\Big)+O\big(L^{-1/2}\big).
\end{align*}
Therefore, by adding-up all the inequalities for $i=1,...,n$ and rearranging terms we obtain
\begin{align*}
\sum_{i=1}^nc_i(M_i-\sqrt{c_i})^2&\leq \sum_{i=1}^nM_i^2\big(E_i(u(t))-E_i(u_0)\big)-\sum_{i=1}^nM_i^2\big(E(\varphi_{c_i})-E_i(u_0)\big)
\\ & \quad -\dfrac{3}{4}\sum_{i=1}^n\big(F_i(u(t))-F(\varphi_{c_i})\big)+O\big(L^{-1/2}\big)=:\mathrm{I}+\mathrm{II}+\mathrm{III}+O(L^{-1/2}).
\end{align*}
Now for the sake of readability we split the proof into three step, each of which is devoted to bound each sum.

\medskip

\textbf{Step 1:} In this first step we are devoted to bound $\mathrm{I}$. First of all notice that by \eqref{lemma_tubular} and the continuous embedding $H^1(\R)\hookrightarrow L^\infty(\R)$ we immediately conclude for all $t\in[0,t^*]$, \[
M_i(t)=c_i+O(\sqrt{\alpha})+O\left(e^{-L/8}\right) \ \,\hbox{ and hence } \ \, 0<M_1<...<M_n.
\]
On the other hand, by using Abel's transformation, the almost monotonicity Lemma \ref{AM_orb_train} and the above estimate, we conclude
\begin{align*}
\mathrm{I}&=M_n^2(t)\sum_{i=1}^n\big(E_i(u(t))-E_i(u_0)\big)-\sum_{j=1}^{n-1}(M_{j+1}^2(t)-M_j^2(t))\sum_{i=1}^j\big(E_i(u(t))-E_i(u_0)\big)
\\ & = \sum_{i=1}^{n-1}\big(M_{i+1}^2(t)-M_i^2(t)\big)\big(\mathcal{I}_{i+1}(t)-\mathcal{I}_{i+1}(0)\big)\leq O\left(e^{-\sqrt{L}}\right),
\end{align*}
which gives us an admissible estimate.

\medskip

\textbf{Step 2:} Now we intend to bound $\mathrm{II}$. In fact, by using the exponential decay of each $\varphi_{c_i}$ and each $\Phi_i$, due to hypothesis \eqref{initial_cond_hyp_train} and by using the reverse triangular inequality we obtain
\begin{align*}
\sum_{i=1}^n \big(E_i(u_0)-E(\varphi_{c_i})\big)&\leq \sum_{i=1}^n \Big\vert\Vert u_0\Vert_{H^1(\mathcal{J}_i(0))}^2-\Vert\varphi_{c_i}(\cdot-x_i(0))\Vert_{H^1(\mathcal{J}_i(0))}^2\Big\vert+O\left(e^{-\sqrt{L}}\right)
\\ & \leq \sum_{i=1}^n \Big(\Vert u_0\Vert_{H^1(\mathcal{J}_i(0))}+\Vert\varphi_{c_i}(\cdot-x_i(0))\Vert_{H^1(\mathcal{J}_i(0))}\Big)\cdot 
\\ & \qquad \cdot \Vert u_0-\varphi_{c_i}(\cdot-x_i(0))\Vert_{H^1(\mathcal{J}_i(0))}+O\left(e^{-\sqrt{L}}\right)
\\ &\leq O(\varepsilon^4)+O\left(e^{-\sqrt{L}}\right),
\end{align*}
Therefore, recalling that $M_i\leq \Vert u_0\Vert_{L^\infty}\leq \Vert u_0\Vert_{H^1}$ we conclude \begin{align*}
\mathrm{II}=-\sum_{i=1}^nM_i^2\big(E(\varphi_{c_i})-E_i(u_0)\big)\leq O(\varepsilon^4)+O\left(e^{-\sqrt{L}}\right).
\end{align*}

\textbf{Step 3:} Finally, to estimate the last term we start by rearranging terms. In fact, notice that each term in $\mathrm{III}$ can be rewritten as 
\begin{align*}
\big\vert F_i(u_0)-F(\varphi_{c_i})\big\vert&=\Big\vert\int \left(u_0^4+2u_0^2u_{0,x}^2-\dfrac{1}{3}u_{0,x}^4\right)(x)\Phi_i(0,x)dx
\\ & \qquad -\int \left(\varphi_{c_i}^4+2\varphi_{c_i}^2\varphi_{c_i}'^2-\dfrac{1}{3}\varphi_{c_i}'^4\right)(\cdot-x_i(0))dx\Big\vert
\\ & \leq \left\vert\int (u_0^2+u_{0,x}^2)^2\Phi_i(0,x)-(\varphi_{c_i}^2+\varphi_{c_i}'^2)^2(\cdot-x_i(0))dx\right\vert
\\ & \qquad + \dfrac{4}{3}\left\vert\int \big( u_{0,x}^4\Phi_i(0,x)-\varphi_{c_i}'^4(\cdot-x_i(0))\big)dx\right\vert =: \mathrm{III}_1+\mathrm{III}_2.
\end{align*} 
For the sake of readability from now on we shall denote by $\varphi_{c_i}=\varphi_{c_i}(\cdot-x_i(0))$. That being said, notice that by using the exponential decay of both $\varphi_{c_i}$ and $\Phi_i$, H\"older's and triangular inequalities together with Sobolev's embedding $H^1\hookrightarrow L^4$ we obtain
\begin{align*}
\mathrm{III}_1&\leq\left\vert\int_{\mathcal{J}_i(0)}\left(u_0^2+u_{0,x}^2+\varphi_{c_i}^2+\varphi_{c_i}'^2\right)\left(u_0^2+u_{0,x}^2-\varphi_{c_i}^2-\varphi_{c_i}'^2\right)dx\right\vert+O\left(e^{-\sqrt{L}}\right)
\\ & \leq 2\Vert u_0^2+u_{0,x}^2+\varphi_{c_i}^2+\varphi_{c_i}'^2\Vert_{L^2\left(\mathcal{J}_i(0)\right)} \Big(\Vert u_0+\varphi_{c_i}\Vert_{L^4\left(\mathcal{J}_i(0)\right)}\Vert u_0-\varphi_{c_i}\Vert_{L^4\left(\mathcal{J}_i(0)\right)}+
\\ & \quad +\Vert u_{0,x}+\varphi_{c_i}'\Vert_{L^4\left(\mathcal{J}_i(0)\right)} \Vert u_{0,x}-\varphi_{c_i}'\Vert_{L^4\left(\mathcal{J}_i(0)\right)} \Big)+O\left(e^{-\sqrt{L}}\right)
\\ & \leq O\big(\Vert u_0-\varphi_{c_i}\Vert_{H^1\left(\mathcal{J}_i(0)\right)}+\Vert u_{0,x}-\varphi_{c_i}'\Vert_{L^4\left(\mathcal{J}_i(0)\right)}\big)+O\left(e^{-\sqrt{L}}\right)
\\ & \leq O\big(\varepsilon^4\big)+O\left(e^{-\sqrt{L}}\right).
\end{align*}
Similarly, due to the exponential decay of $\varphi_{c_i}$ and $\Phi_i$, by using H\"older's and triangular inequalities we get
\begin{align*}
\mathrm{III}_2&\leq\dfrac{4}{3}\left\vert\int_{\mathcal{J}_i(0)} \big(u_{0,x}^4-\varphi_{c_i}'^4\big)dx\right\vert+O\left(e^{-\sqrt{L}}\right)
\\ & =\dfrac{4}{3}\left\vert\int_{\mathcal{J}_i(0)}\big(u_{0,x}^2+\varphi_{c_i}'^2\big)\big(u_{0,x}+\varphi_{c_i}'\big)\big(u_{0,x}-\varphi_{c_i}'\big)dx\right\vert+O\left(e^{-\sqrt{L}}\right)
\\ & \leq 2\Vert u_{0,x}-\varphi_{c_i}'\Vert_{L^4\left(\mathcal{J}_i(0)\right)}^3\Vert u_{0,x}-\varphi_{c_i}'\Vert_{L^4\left(\mathcal{J}_i(0)\right)}+O\left(e^{-\sqrt{L}}\right)
\\ & \leq O(\varepsilon^4)+O\left(e^{-\sqrt{L}}\right).
\end{align*}
Adding-up all the previous estimates we conclude the proof of the lemma.
\end{proof}

\subsection{Proof of Theorem \ref{MT2}}

First of all, notice that by using \eqref{initial_cond_hyp_train} and the reverse triangular inequality we deduce
\begin{align*}
\left\vert E(u_0)-E\left(R_{\vec{z}_0}\right)\right\vert&=\left(\Vert u_0\Vert_{H^1}+\Vert R_{\vec{z}_0}\right\Vert _{H^1})\Big\vert \Vert u_0\Vert_{H^1}-\Vert R_{\vec{z}}\Vert_{H^1}\Big\vert
\\ & \leq \left(\Vert u_0\Vert_{H^1}+\Vert R_{\vec{z}_0}\right\Vert _{H^1}) \Vert u_0- R_{\vec{z}}\Vert_{H^1} =O(\varepsilon^4).
\end{align*}
On the other hand, by using Lemma \ref{lemm_four_three} together with Lemma \ref{lem_four_six} as well as the previous estimate, recalling that $E\big(R_{\vec{z}_0}\big)=\sum_{i=1}^n E(\varphi_{c_i})+O(\exp(-L/4))$, we obtain
\begin{align*}
\left\Vert u(t^*)-\sum_{i=1}^n\varphi_{c_i}\big(\cdot-x_i(t^*)\big)\right\Vert_{H^1}^2&=E(u_0)-\sum_{i=1}^n E(\varphi_{c_i})
\\ & \quad -4\sum_{i=1}^n\sqrt{c_i}\big(M_i(t^*)-\sqrt{c_i}\big)+O\left(e^{-\sqrt{L}}\right)
\\ & = O\left(\varepsilon^4\right)+O\left(\varepsilon^2\right)+O\left(L^{-1/4}\right)=O\left(\varepsilon^2\right)+O\left(L^{-1/4}\right).
\end{align*}
In other words, there exists $\widetilde{C}>0$ such that \[
\left\Vert u(t^*)-\sum_{i=1}^n\varphi_{c_i}\big(\cdot-x_i(t^*)\big)\right\Vert_{H^1}^2\leq \widetilde{C}\left(\varepsilon^2+L^{-1/4}\right).
\]
Therefore, by taking $C$, the constant appearing in \eqref{neigh_assumption}, so that $C^2=4\widetilde{C}$ we conclude the proof of the theorem. \qed

\subsection{Proof of Corollary \ref{cor_MT2}}\label{sub_cor_multipeakon}
As we mentioned in the introduction, the Novikov equation \eqref{nov_eq_2} possesses multi-peakon-antipeakon solutions given by \[
u(t,x)=\sum_{i=1}^np_i(t)e^{-\vert x-q_i(t)\vert}, \quad n\in\N
\]
where the pairs $(p_i,q_i)\in\R^2$ satisfy the Hamiltonian system of ODEs: \begin{align*}
\begin{cases}
\dfrac{dq_i}{dt}=u^2(q_i)=\displaystyle\sum_{j,k=1}^np_jp_ke^{-\vert q_i-q_j\vert-\vert q_i-q_k\vert},
\\ \displaystyle\dfrac{dp_i}{dt}=-p_iu(q_i)u_x(q_i)=p_i\sum_{j,k=1}^np_jp_k\sgn(q_i-q_j)e^{-\vert q_i-q_j\vert-\vert q_i-q_k\vert}.\end{cases}
\end{align*}
It is easy to check that the local solutions of this differential system can be uniquely extended as long as the $q_i$'s remain different from each other. In fact, if for some time $t^*$ and some $i\neq j$ we have $q_i(t^*)=q_j(t^*)$, then  uniqueness fails and this breakdown leads to the usually subtle question regarding continuation of solutions beyond the collision. In the Camassa-Holm case this latter question is rather well-understood (see for instance \cite{BrCo,BrCo2,HoRa,HoRa2}). However, in \cite{HoLuSz} Theorem $4.5$, Hones et al. proved that if  at the initial time all $p_i$'s are positives, i.e. there are only peakons, then all $q_i$'s stay different from each other for all times. Of course, the case where there are only antipeakons also holds, however this is not longer true if we allow the existence of peakons and antipeakons simultaneously. More precisely, Hones et. al. proved that if at the initial time \begin{align}\label{initial_hyp_cor_p_i}
p_1^0,...,p_n^0>0 \quad \hbox{and}\quad q_1^0<...<q_n^0,
\end{align}
then, \eqref{initial_hyp_cor_p_i} holds for all times $t\in\R$. In particular, under these hypothesis the different peakons never overlaps each other. For example, if a large peak follows a smaller one, due to its different speeds, they shall eventually get close enough so that the larger one shall transfer part of its energy to the smaller one. Then, the smallest shall become the largest and both peakons shall be well ordered.

\medskip

Regarding the asymptotics of $(p_i,q_i)(t)$, in \cite{HoLuSz}  Hones et al. also proved that under these hypotheses the following equalities hold:
\begin{align}\label{asymptotics_p_q_i}
\lim_{t\to+\infty} p_i^2(t)=\lim_{t\to+\infty}\dot{q}_i(t)=\lambda_i^2 \quad \hbox{ and }\quad 
\lim_{t\to-\infty} p_i^2(t)=\lim_{t\to-\infty}\dot{q}_i(t)=\lambda_{n+1-i}^2,
\end{align}
where we recall that the values $\lambda_i$ correspond to the square roots of the eigenvalues of the matrix $TPEP$ defined in the statement of Corollary \ref{cor_MT2}.

\medskip

Now, let $\delta>0$ small but fixed and let us consider $\big(p_i(0),q_i(0))$ satisfying \eqref{initial_hyp_cor_p_i}. Hence, from the asymptotics of $p_i(t)$ and $q_i(t)$ we deduce the existence of a time $T\gg1$ sufficiently large such that for all $i=2,...n$ we have \[
q_i(T)-q_{i-1}(T)\geq L \quad \hbox{and}\quad q_{i-1}(-T)-q_i(-T)\geq L,
\]
with $L$ being any positive number satisfying $L>2\max\big\{L_0,\left(\tfrac{\delta}{\textbf{A}}\right)^{-8}\big\}$, where $\textbf{A}=2\max\{1,A\}$ and $A>0$ is the implicit constant involved in \eqref{orb_concl}. Thus, by using the second item in Theorem \ref{theorem_gwp} we deduce that there exists $\varepsilon>0$ small enough only depending on $T$ and $\delta$  such that for any initial data $u_0\in Y_+(\R)$ satisfying \eqref{hyp_cor_orb} the following holds: For all $t\in[-T,T]$ we have
\begin{align}\label{bound_compact_time}
\left\Vert u(t)-\sum_{i=1}^np_i(t)e^{-\vert x-q_i(t)\vert}\right\Vert_{H^1}\leq \left(\dfrac{\delta}{2\textbf{A}}\right)^4.
\end{align}
On the other hand, due to the fact that the Novikov equation \eqref{nov_eq_2} is invariant under the transformation $(t,x)\mapsto(-t,-x)$, Theorem \ref{MT2} also holds when replacing $t$ by $-t$, $z_i$ by $-z_i$ and $x_i(t)$ by $-x_i(-t)$. This give us the same stability result backwards in time for a train of peakons that are ordered in the inverse order than in the statement of Theorem \ref{MT2}. Therefore, by gathering \eqref{initial_hyp_cor_p_i} with \eqref{bound_compact_time} together with Theorem \ref{MT2} we conclude \eqref{first_part_cor}.

\medskip

Now, to prove the second part of the corollary, it is enough to notice that by using \eqref{asymptotics_p_q_i}, by making $T$ bigger if necessary, without loss of generality we may also assume that \[
\vert p_i(T)-\lambda_i\vert \leq \dfrac{1}{100n}\left(\dfrac{\delta}{\textbf{A}}\right)^4 \quad \hbox{and}\quad \vert p_i(-T)-\lambda_{n+1-i}\vert \leq \dfrac{1}{100n}\left(\dfrac{\delta}{\textbf{A}}\right)^4.
\]
Thus, by using \eqref{bound_compact_time} we obtain that \[
\left\Vert u(T,\cdot)-\sum_{i=1}^n\lambda_ie^{-\vert x-q_i(T)\vert}\right\Vert_{H^1}+
\left\Vert u(-T,\cdot)-\sum_{i=1}^n\lambda_{n+1-i}e^{-\vert x-q_i(-T)\vert}\right\Vert_{H^1}\leq \left(\dfrac{\delta}{\textbf{A}}\right)^4.
\]
Hence, we conclude by using Theorem \ref{MT2}. The proof is complete. \qed

\medskip

\section{Asymptotic stability of a train of peakons}\label{sec_MT5}

In this section we aim to prove Theorem \ref{MT5}. Notice that gathering this latter result with the asymptotics for multipeakons stated in subsection \ref{sub_cor_multipeakon} and Corollary \ref{cor_MT2}, we immediately obtain Corollary \ref{cor_MT5}. 

\medskip

From now on we shall follow Molinet's ideas (see \cite{Mo}) for the Camassa-Holm equation, which are based on the proof of asymptotic stability for the sum of $n$-solitons for the gKdV equation (see \cite{MaMeTs}).

\subsection{Almost monotonicity lemma}\label{sec_six_one}

In the rest of this paper we shall need to explicitly study the behavior of the solution $u(t)$ on both, the left and right part of the space. We recall that the weight function $\Psi$ is given by \begin{align}\label{psi_def_2}
\Psi(x):=\dfrac{2}{\pi}\arctan\left(\exp(\tfrac{x}{6})\right), \quad \hbox{so that}\quad \Psi(x)\to 1 \ \hbox{ as }\ x\to+\infty.
\end{align}
Now, let us fix (for the rest of this paper) some parameters: \begin{align}\label{parameters}
\varepsilon^\star:=\left(\dfrac{\min\{1,\sigma\}}{2^{18}}\right)^8, \ \  L_0:=\left(2^{18}\max\{1,\sigma^{-1}\}\right)^{32},
\end{align}
and $\sigma:=\mathbf{C}\min\{c_2-c_1,...,c_n-c_{n-1},\beta\}$, where $\mathbf{C}:=\min\{1,\widetilde{C}^{-1}\}$ and $\widetilde{C}>0$ is the implicit constant involved in \eqref{orb_concl}.
\begin{lem}[Almost monotonicity of the energy]\label{tech_lem_mon_exp}
Assume that we are under the hypothesis and notations of Theorem \ref{MT5} and Lemma \ref{mod_lemma}. Additionally, set $\delta_1,...,\delta_n\in(0,1)$ and define the family of differentiable functions $z_1,...,z_n:\R\to\R$ as follows \[
\delta_1:=1-\tfrac{\beta}{4c_1}, \quad z_1(t):=\tfrac{\beta}{2}t \quad \hbox{and} \quad \delta_i:=\tfrac{5}{8}(c_i-c_{i-1}), \quad  z_i(t):=(1-\delta_i)\widetilde{x}_i(t).
\]
Then, there exists $R_0>0$ sufficiently large such that for all $t\in\R$ it holds 
\begin{align}\label{condition_tail}
\Vert u(t)\Vert_{L^\infty(x-\widetilde{x}_n(t)>R_0)}\leq \dfrac{(1-\delta_n)c_n}{\mathbf{b}},\ \hbox{ where } \ \mathbf{b}:=2^6\max\{1,\Vert u_0\Vert_{H^1}\}.
\end{align}
Moreover, for any $R>R_0$ the following property holds: For each $i=1,...,n$ there exists $t_R^i>0$ only depending on $R$ such that for any $t_0^i\geq t_R^i$, defining the modified energy functionals
\begin{align*}
\mathrm{I}_{i,t_0^i}^{\pm R}:=\int\big(u^2+u_x^2\big)(t,x)\Psi\big(\cdot-z_{i}^{\pm R}(t)\big)dx  \ \,\hbox{ where } \ \, z_{i}^{\pm R}(t):=\widetilde{x}_i(t_0^i)\pm R+z_i(t)-z_i(t_{0}^i),
\end{align*}
the following estimates hold:
\begin{align}\label{AM_right_n_energy}
\forall t\leq t_0^n, \ \, \mathrm{I}_{n,t_0^n}^R(t_0^n)-\mathrm{I}_{n,t_0^n}^R(t)\leq Ce^{-R/6}, \ \hbox{ and } \,\ \forall t\geq t_0^n, \, \ \mathrm{I}_{n,t_0^n}^{-R}(t)-\mathrm{I}_{n,t_0^n}^{-R}(t_0^n)\leq Ce^{-R/6}.
\end{align}
Moreover, for any $i=1,...,n-1$ we have \begin{align}\label{AM_right_i_energy}
\mathrm{I}_{i,t_0^i}^{-R}(t)-\mathrm{I}_{i,t_0^i}^{-R}(t_0^i)\leq Ce^{-R/24}, \ \hbox{ for all } t\geq t_0^i.
\end{align}
\end{lem}
\begin{proof}
The proof is somehow contained in the proof of Lemma \ref{AM_orb_train}, which at the same time is somehow contained in the proof of Lemma $3.2$ in \cite{Pa}. However, for the sake of completeness we show a sketch of the proof in the appendix. See Section \ref{sec_ap_tech_lem}.
\end{proof}
Now, before going further we shall need to introduce some additional notation. For $v\in Y$ and $R>0$ we define the functionals $\mathcal{J}_l^R$ and $\mathcal{J}_r^R$ given by \begin{align*}
\mathcal{J}_r^R:=\left\langle v^2+v_x^2,\Psi(\cdot-R)\right\rangle \ \hbox{ and }\ \, \mathcal{J}_l^R:=\left\langle v^2+v_x^2,1-\Psi(\cdot+R)\right\rangle.
\end{align*}
Now, under the notations of the previous lemma notice that, for this choice of parameters and from the definitions of $\mathcal{J}_r^R$ and $\mathrm{I}_{i}^R$ we immediately obtain that \[
\mathcal{J}_{n,r}^R(t):=\mathcal{J}_r^R\big(u(t,\cdot+\widetilde{x}_n(t))\big)\geq \mathrm{I}_{n}^R(t), \quad \forall t\leq t_0^n,
\]
where $\mathrm{I}_{n}^R(t)$ is the functional defined in Lemma \ref{tech_lem_mon_exp}. Moreover, notice that in particular we have $\mathcal{J}_{n,r}^R\big(t_0^n\big)=\mathrm{I}_{n}^R(t_0^n)$. Thus, by using \eqref{AM_right_n_energy} we obtain \begin{align}\label{ineq_J_r}
\mathcal{J}_{r}^R\big(u(t_0^n,\cdot+\widetilde{x}_n(t_0^n))\big)\leq \mathcal{J}_{r}^R\big(u(t,\cdot+\widetilde{x}_n(t))\big)+C e^{-\frac{R}{6}}, \quad \forall t\leq t_0^n,
\end{align}
where $C>0$ is the constant appearing in \eqref{AM_right_n_energy}. On the other hand,  for the sake of notation we also introduce the functional $\widetilde{\mathrm{I}}_{i}^R(t)$ given by
\[
\widetilde{\mathrm{I}}_{i}^R(t):=\left\langle u^2+u_x^2,1-\Psi\big(\cdot-\delta_i\widetilde{x}_i(t_0^i)+R-(1-\delta_i)\widetilde{x}_i(t)\big)\right\rangle=E(u)-\mathrm{I}_{i}^{-R}(t),
\]
where the parameter $\delta_i>0$ is defined in Lemma \ref{tech_lem_mon_exp}. Notice that due to the energy conservation together with inequality \eqref{AM_right_n_energy} we deduce \begin{align}\label{a_m_left_energy}
\widetilde{\mathrm{I}}_{i}^R(t)\geq \widetilde{\mathrm{I}}_{i}^R(t_0^i)-Ce^{-R/6}.
\end{align}
Therefore, from the energy conservation and the previous inequality we deduce that for all $i=1,...,n$ and all $t\geq t_0^i$ we have \begin{align}\label{ineq_J_l}
\mathcal{J}_l^R\big(u(t,\cdot+\widetilde{x}_i(t))\big)\geq \mathcal{J}_l^R\left(u\left(t_0^i,\cdot+\widetilde{x}_i(t_0^i)\right)\right)-Ce^{-\frac{R}{6}}.
\end{align}
The case of $\mathcal{J}_{i,r}^R$ is more subtle and its proof is the aim of the following lemma.
\begin{lem}\label{tech_lem_left}
Assume we are under the hypothesis and notations of Lemma \ref{tech_lem_mon_exp}. For $i=1,...,n-1$ define the following modified energy functionals \begin{align*}
\mathcal{J}_{i,r}^R(t):=\int \big(u^2+u_x^2\big)(t,x)\Psi\big(\cdot-\widetilde{x}_i(t)-R\big)dx.
\end{align*}
Then, for any $R>0$ and all pair $(t,t_0)$ satisfying $t_R^{i+1}\leq t\leq t_0$ the following inequality holds: \begin{align}\label{ineq_i_J_r}
\mathcal{J}_{i,r}^R(t_0)\leq \mathcal{J}_{i,r}^R(t)+Ce^{-R/24},
\end{align}
where $\{t_R^i\}_{i=1}^n$ are defined in the proof of Lemma \ref{tech_lem_mon_exp} (see \eqref{t_i_R_def}).
\end{lem}

\begin{proof}
See the appendix, Section \ref{ap_tech_left}.
\end{proof}

\subsection{End of the proof of Theorem \ref{MT5}}

The following property is the analogous convergence result for a single peakon in the case of a train of peakons. 
\begin{lem}\label{convergence_result}
For every $i=1,...,n$ the following strong convergence holds: \begin{align}\label{local_strong_conv}
u\big(t,\cdot+\widetilde{x}_i(t)\big)-\rho_i(t)\varphi\to 0 \ \hbox{ in } \ H^1_{loc}(\R) \ \hbox{ as } \ t\to+\infty,
\end{align}
where $\rho_i(t):=u\big(t,x_i(t)\big)$, i.e. the maximum of $u(t)$ over the sets $\mathcal{J}_i(t)$ defined in Lemma \ref{mod_lemma}. Moreover, in the case $i=n$ the following holds: For every $A>0$ we have \begin{align}\label{conv_two}
u\big(t,\cdot+\widetilde{x}_n(t)\big)-\rho_n(t)\varphi\to0 \ \hbox{ in } \ H^1\big((-A,+\infty)\big) \ \hbox{ as } \ t\to+\infty.
\end{align}
\end{lem}

\begin{proof}
This is a consequence of Proposition $4.2$ in \cite{Pa} and the second inequality in \eqref{mod_parameter_bound}. In fact, notice that the proof of that proposition only requires Lemmas \ref{tech_lem_mon_exp} and \ref{tech_lem_left} of the present work to hold. Since the proof follows exactly the same lines as the ones in Proposition $4.2$ in \cite{Pa}, we only sketch it for the case $i=n$. In fact, following exactly the same lines it is possible to show that for any increasing sequence $t_n\to+\infty$ there exists a subsequence $\{t_{n_k}\}$ and a function $u_0^\star\in Y_+$ such that as $k\to+\infty$ we have \begin{align}\label{weak_strong_conv}
u\big(t_{n_k},\cdot+\widetilde{x}_n(t_{n_k})\big)\rightharpoonup u_0^\star  \ \hbox{ in } \ H^1(\R) \quad \hbox{and}\quad u\big(t_{n_k},\cdot+\widetilde{x}_n(t_{n_k})\big)\to u_0^\star \ \hbox{ in } \ H^{1^-}_{loc}(\R).
\end{align}
Then, by using the almost monotonicity inequalities \eqref{ineq_J_r} and \eqref{ineq_J_l} we can prove that the weak solution to equation \eqref{nov_eq_2} associated to $u_0^\star$ is actually an $H^1$-almost localized solution, and hence it is a peakon (c.f. Theorem $1.3$ in \cite{Pa}). Therefore, there exists $x_0\in\R$ and $c_n^*>0$ such that \[
u_0^\star=\varphi_{c_n^*}(\cdot-x_0).
\] 
On the other hand, notice that due to the local strong $L^2$ convergence we deduce that for all $K\subset \R $ compact we have \[
\lim_{k\to+\infty}\Vert u(t_{n_k},\cdot+\widetilde{x}_n(t_{n_k}))-\varphi_{c_n^*}\Vert_{L^2(K)}=0.
\]
On the other hand, due to the fact that $\vert v_x\vert\leq v$ for any $v\in Y_+$ we deduce \[
\liminf_{k\to+\infty}\Vert u_x(t_{n_k},\cdot+\widetilde{x}_n(t_{n_k}))\Vert_{L^2(K)}\leq \lim_{k\to+\infty}\Vert u(t_{n_k},\cdot+\widetilde{x}_n(t_{n_k}))\Vert_{L^2(K)}=\Vert \varphi_{c_n^*}\Vert_{L^2(K)}
\]
Hence, by using again that $\Vert \varphi'\Vert_{L^2(K)}=\Vert \varphi\Vert_{L^2(K)}$ we obtain \[
\liminf_{k\to+\infty}\Vert u(t_{n_k},\cdot+\widetilde{x}_n(t_{n_k}))\Vert_{H^1(K)}^2\leq2\Vert \varphi_{c_n^*}\Vert_{L^2(K)}^2=\Vert \varphi_{c_n^*}\Vert_{H^1(K)}^2,
\]
Thus, by a standard result in Functional Analysis we know that the weak convergence result together with the previous inequality implies that \begin{align}\label{strong_h1}
u(t_{n_k},\cdot+\widetilde{x}_n(t_{n_k}))-\varphi_{c_n^*}\to0 \ \hbox{ in }\  H^1_{loc} \  \hbox{ as } \  k\to+\infty.
\end{align}
Finally, let us prove the strong $H^1$ convergence in $(-A,\infty)$ for any fixed $A>0$. In fact, first of all, notice that the weak convergence result \eqref{weak_strong_conv} together with the uniform estimate \eqref{mod_bound} and the definition of $\varepsilon^\star$ implies that \[
\Vert \varphi_{c_n^\star}(\cdot-x_0)-\varphi_{c_n}\Vert_{H^1}\leq C\varepsilon^\star \quad \hbox{and} \quad 
\vert c_n-c_n^*\vert\ll \sigma,
\]
and hence, by using the local strong convergence \eqref{strong_h1} we infer that $\vert x_0\vert \ll 1$. On the other hand, notice that the weak convergence result \eqref{weak_strong_conv} forces $u_0^\star$ to satisfy the $n$-th orthogonality condition \eqref{mod_orthogonality}. Therefore, by using \eqref{orth_cond_def} we obtain that $x_0$ has to be equal to zero. Finally, notice that the convergence result \eqref{strong_h1} together with \eqref{mod_bound} implies that \[
\sqrt{c_n^*}=\lim_{k\to+\infty}\max_{\mathcal{J}_n(t_{n_k})}u(t_{n_k}).
\]
Thus, defining $\rho_n(t):=\max\{ u(t,\cdot): \ x\in\mathcal{J}_n(t)\} $ we deduce that as $k\to+\infty$ we have \[
u(t_{n_k},\cdot+\widetilde{x}_n(t_{n_k}))-\rho_n(t_{n_k})\varphi\rightharpoonup0 \ \hbox{ in }\ H^1.
\]
Since this is the only possible limit we conclude that as $t\to+\infty$ we have \begin{align}\label{weak_conv_H1_peakon}
u(t,\cdot+\widetilde{x}_n(t))-\rho_n(t)\varphi\rightharpoonup 0 \  \hbox{ in } \ H^1 \ \ \hbox{ and } \ \ u(t,\cdot+\widetilde{x}_n(t))-\rho_n(t)\varphi\to 0 \  \hbox{ in } \ H^{1}_{loc}.
\end{align}
Now, we claim that the latter convergence result implies that for any fixed $A>0$, as $t\to+\infty$, the following convergence holds: \begin{align}\label{local_strong_h1}
u(t,\cdot+\widetilde{x}_n(t))-\rho_n(t)\varphi\to0  \ \hbox{ in } \ H^1((-A,\infty)).
\end{align}
In fact, let $\delta>0$ be fixed and consider $R\gg1$ sufficiently large such that \[
\mathcal{J}_r^R\big(u(0,\cdot+\widetilde{x}_n(0)\big)<\delta \quad  \hbox{ and } \quad Ce^{-R/6}<\delta,
\] 
where $C>0$ is the constant involved in \eqref{ineq_J_r}. Then, from the almost decay of the energy at the right \eqref{ineq_J_r} we infer that \[
\mathcal{J}_r^R\big(u(t,\cdot+\widetilde{x}_n(t))\big)<2\delta, \ \hbox{ for all } \, t\in\R.
\]
Nevertheless, the latter inequality together with the local strong convergence in $H^1$ given in \eqref{local_strong_h1} immediately implies that, for any $A>0$ we have \begin{align}\label{H1_conv_right}
u(t,\cdot+\widetilde{x}_n(t))-\rho_n(t)\varphi\xrightarrow{t\to+\infty}0 \ \hbox{ in } \ H^1((-A,\infty)).
\end{align}
The sketch of the proof is complete.
\end{proof}

\textbf{Important:} Notice that in the same fashion as in \eqref{weak_strong_conv}, following the same lines in Proposition $4.2$ in \cite{Pa} and by using the rigidity Theorem we deduce that for any increasing sequence $t_n\to+\infty$ the exists a subsequence $t_{n_k}$ and positive numbers $c_1^*,...,c_n^*$ such that \begin{align}
u\big(t_{n_k},\cdot+\widetilde{x}_i(t_{n_k})\big)-\varphi_{c_i^*}\to 0 \ \hbox{ in } \ H^1_{loc}(\R) \ \hbox{ as } \ k\to+\infty,
\end{align} 
Now, before going further let us introduce some notation. From now on and for $i=1,...,n$ we shall denote by $W_i$ and $w_i$ the functions defined by
\begin{align}\label{def_W_w_i}
W_i:=\sum_{j=i}^n \sqrt{c_j^*}\varphi(\cdot-\widetilde{x}_j(t))\quad \hbox{and}\quad w_i:=u-W_i.
\end{align}
Additionally, we define the following modified energy functional \[
\widetilde{\mathcal{I}}(t):=\int\left(u^2+u_x^2\right)(t,\cdot)\Psi\big(\cdot-y_n(t)\big)dx,
\]
which corresponds to $\mathrm{I}_n^{y_n(t_{R}^n)-x_n(t_R^n)}(t)$ by redefining $z_n(\cdot)$ as $y_n(\cdot)$ in Lemma \ref{tech_lem_mon_exp}, where $\{y_i\}_{i=1}^n$ are defined in Lemma \ref{mod_lemma} and we change $y_1(t):=\tfrac{\beta}{2}$. Now, notice that for $t\geq t_R^n$ we have that $x_n(t_R^n)-y_n(t_R^n)\geq R\geq R_0$, and hence, by the same proof as in Lemma \ref{tech_lem_mon_exp} we deduce that $\widetilde{\mathcal{I}}$ is almost non-increasing for $t\geq t_R$. 

\medskip

The following two lemmas about the convergences of $\rho_i(t)$ and $\dot{\widetilde{x}}_i(t)$ in the case of the fastest peakon (i.e. $i=n$) follow the same lines as the ones for the single peakon case (c.f. \cite{Pa}, see also \cite{Mo}). However, for the sake of completeness we prove it anyway.

\begin{lem}\label{lem_convergence_rho_n}
As $t$ goes to $+\infty$ the following convergence holds:
\[
\rho_n(t)\to \sqrt{c_n^*}.
\]
In particular, as a consequence of the previous convergence, as $t$ goes to $+\infty$ we also have:
\[
\int \left(\big(u^2-\sqrt{c_n^*}\varphi(\cdot-\widetilde{x}_n(t))\big)^2+\big(u_x^2-\sqrt{c_n^*}\varphi'(\cdot-\widetilde{x}_n(t))\big)^2\right)\Psi\big(\cdot-y_n(t)\big)\to 0.
\]
\end{lem}
\begin{proof}
In fact, let $\epsilon>0$ arbitrarily small but fixed and consider $R\gg1$ sufficiently large such that $Ce^{-R/6}<\epsilon$. Then, by using \eqref{ineq_J_l} as well as the energy conservation we obtain that for all $t>t'>t_0^n$ we have \[
\int \big(u^2+u_x^2\big)(t)\Psi\big(x-\widetilde{x}_n(t)+R\big)\leq\epsilon+ \int \big(u^2+u_x^2\big)(t')\Psi\big(x-\widetilde{x}_n(t')+R\big).
\]
On the other hand, due to the strong convergence result \eqref{conv_two} and the exponential localization of both $\varphi$ and $\Psi$, we infer that there exists $t_0\gg1$ sufficiently large such that for all $t\geq t_0$ we have \[
\left\vert\int \big(u^2+u_x^2\big)(t)\Psi(x-\widetilde{x}_n(t)+R)-\rho_n^2(t)E(\varphi)\right\vert\leq\epsilon.
\]
Plugging the last two inequalities together we conclude that for any pair of times $(t,t')\in\R^2$ satisfying $t>t'>\max\{t_0,t_0^n\}$ we have \[
\rho_n^2(t)E(\varphi)\leq \rho_n^2(t')E(\varphi)+3\epsilon.
\]
Since $\epsilon>0$ was arbitrary, the latter inequality forces $\rho(t)$ to have a limit at $+\infty$ and thus to converge to \[
\lim_{t\to+\infty}\rho_n(t)=\sqrt{c_n^*}
\]
what finish the proof of the lemma.
\end{proof}

\begin{lem}\label{lema_convergence_x_dot_n}
As $t$ goes to $+\infty$ the following convergence holds $\dot{\widetilde{x}}_n(t)\to c_n^*$.
\end{lem}
\begin{proof}
First of all, let us start by recalling and introducing some notation:  \[
w_1(t)=u-\sum_{j=1}^n\varphi_{c_j^*}(\cdot-\widetilde{x}_j(t)), \ \ \omega_i:=\varphi_{c_i^*}(\cdot-\widetilde{x}_i(t)) \  \hbox{ and } \ \omega_{n_0}^i:=(\rho_{n_0}*\varphi_{c_i^*})(\cdot-\widetilde{x}_i(t)).
\]
Now, on the one-hand, by differentiating the $n$-th equation in  \eqref{mod_orthogonality} with respect to time and by using that $\varphi$ satisfies the equation $\varphi-\varphi''=2\delta$ we obtain \[
\int w_{1,t}\omega_{n_0,x}^n=\dot{\widetilde{x}}_n\int w_1(t,x)\omega_{n_0}^n(t,x)dx-2\sqrt{c_n^*}\dot{\widetilde{x}}_n\int w_1(t,x)\rho_{n_0}\big(x-\widetilde{x}_n(t)\big)dx.
\]
On the other hand, by using that $\varphi$ solves \eqref{nov_eq_2} we infer that each $\omega_i(t,x)$ satisfy the following equation:
\[
\omega_{i,t}+(\dot{\widetilde{x}}_i-c_i^*)\omega_{i,x}+\omega_i^2\omega_{i,x}=p_x*\Big(\omega_i^3+\dfrac{3}{2}\omega_i\omega_{i,x}^2\Big)-\dfrac{1}{2}p*\omega_{i,x}^3
\]
Therefore, using that $u(t)$ also solves \eqref{nov_eq_2}, by replacing $u=w_1+\sum_{i=1}^n\omega_i$ and then using the equation satisfied by each $\omega_i$ we obtain
\begin{align}\label{mod_huge_eq_11}
w_{1,t}-\sum_{j=1}^n(\dot{\widetilde{x}}_j-c_j^*)\omega_{j,x}&=-\left(w_1+\sum_{j=1}^n\omega_j\right)^2w_{1,x}-w_1^2\sum_{j=1}^n\omega_{j,x}-2w_1\sum_{j,k=1}^n\omega_j\omega_{k,x}\nonumber
\\ & \quad -\sum_{\substack{j,k,\ell=1 \\ (k,\ell)\neq(j,j)}}^n\omega_j\omega_k\omega_{\ell,x}-p_x*w_1^3-3\sum_{j=1}^np_x*(w_1^2\omega_j)\nonumber
\\ &  \quad  -3\sum_{j,k=1}^np_x*w_1\omega_j\omega_k-\sum_{\substack{j,k,\ell=1 \\ (k,\ell)\neq(j,j)}}^np_x*\omega_j\omega_k\omega_\ell \nonumber
\\ & \quad -\dfrac{3}{2}p_xw_1w_{1,x}^2-3\sum_{j=1}^np*w_1w_{1,x}\omega_{j,x}-\dfrac{3}{2}\sum_{j,k=1}^np*w_1\omega_{j,x}\omega_{k,x}\nonumber
\\ & \quad -\dfrac{3}{2}\sum_{j=1}^np*w_{1,x}^2\omega_j-3\sum_{j,k=1}^np*w_{1,x}\omega_j\omega_{k,x}\nonumber
\\ & \quad -\dfrac{3}{2}\sum_{\substack{j,k,\ell=1 \\ (k,\ell)\neq(j,j)}}^np_x*\omega_j\omega_{k,x}\omega_{\ell,x}-\dfrac{1}{2}p*w_{1,x}^3-\dfrac{3}{2}\sum_{j=1}^np*w_{1,x}^2\omega_{j,x}\nonumber
\\ & \quad -\dfrac{3}{2}\sum_{j,k=1}^n p*w_{1,x}\omega_{j,x}\omega_{k,x}-\dfrac{1}{2}\sum_{\substack{j,k,\ell=1 \\ (k,\ell)\neq(j,j)}}^n p*\omega_{j,x}\omega_{k,x}\omega_{\ell,x}.
\end{align}
On the other hand, notice that due to \eqref{conv_two}, inequality \eqref{mod_parameter_bound} and the exponential decay of both $\omega_i$ and $\omega_{n_0}^n$ we deduce that $\Vert \omega_i \omega_{n_0}^n\Vert_{L^1}+\Vert \omega_{i,x}\omega_{n_0}^n\Vert_{L^1}\to 0$ as $t\to+\infty$ for $i\neq n$ and
\[
\Vert w_1^2\omega_{n_0,x}^n\Vert_{L^1}+\Vert w_{1,x}^2\omega_{n_0,x}^n\Vert_{L^1}+\int \vert w_1\omega_{n_0}^n\vert dx+\int \vert w_1\rho_{n_0}(x-x(t))\vert dx\to0 \ \hbox{ as }\ t\to+\infty.
\]
Therefore, by taking the $L^2$-inner product from equation \eqref{mod_huge_eq_11} against $\omega_{n_0,x}^n$ and noticing that $\langle \omega_{n,x}(t),\omega_{n_0,x}^n(t)\rangle_{L^2,L^2}\equiv \mathrm{constant}>0$ for all times $t\in\R$ we conclude \[
\dot{\widetilde{x}}_n-c_n^*\to 0 \ \hbox{ as }\ t\to+\infty.
\]
The proof is complete.
\end{proof}

Finally, it only remains to prove the analogous properties to Lemmas \ref{lem_convergence_rho_n} and \ref{lema_convergence_x_dot_n} for the cases $i=1,...,n-1$. This is the aim of the remaining part of this subsection. 

\medskip

\textbf{Inductive argument:}  Now we proceed by an inductive argument, that is, from now on we assume that for some $i^\star\in\{1,...,n-2\}$ it holds that \begin{align}\label{inductive_hypothesis}
&\int \left(u^2-\sum_{j=i^\star+1}^n\sqrt{c_j^*}\varphi(\cdot-\widetilde{x}_j(t))\right)^2\Psi\big(\cdot-y_{i^\star+1}(t)\big)dx\nonumber
\\ & \qquad +\int\left(u_x^2-\sum_{j=i^\star+1}^n\sqrt{c_j^*}\varphi'(\cdot-\widetilde{x}_j(t))\right)^2\Psi\big(\cdot-y_{i^\star+1}(t)\big)dx\to 0
\end{align}
and we intend to prove that, as $t$ goes to $+\infty$, this implies that $\dot{\widetilde{x}}_i(t)\to c_i^*$, $\rho_i(t)\to \sqrt{c_i^*}$ and  \begin{align}\label{inductive_conclusion}
&\int \left(u^2-\sum_{j=i^\star}^n\sqrt{c_j^*}\varphi(\cdot-\widetilde{x}_j(t))\right)^2\Psi\big(\cdot-y_{i^\star}(t)\big)dx\nonumber
\\ & \qquad +\int\left(u_x^2-\sum_{j=i^\star}^n\sqrt{c_j^*}\varphi'(\cdot-\widetilde{x}_j(t))\right)^2\Psi\big(\cdot-y_{i^\star}(t)\big)dx\to 0.
\end{align}
For the sake of simplicity we split the proof into $6$ steps where only the first five of them are devoted to prove the inductive argument. First, we start by proving some extra monotonicity property. The second step intends to state the analogous to the convergence result \eqref{conv_two} in the case $i\neq n$. Then, we prove the convergences of the scaling and velocity parameters. In step five we intend to conclude the inductive argument by proving \eqref{inductive_conclusion}. Finally, in the last step we are devoted to conclude the convergence result \eqref{MT_5_conclusions} on the first set in $\mathcal{A}_t$. For the sake of simplicity from now on we drop the superindex in $i^\star$ and hence we just denote it by $i$. 

\medskip

\textbf{Step 1:} Let $w_{i+1}:=u-W_{i+1}$ (see \eqref{def_W_w_i} for the definition of $\{W_i\}_{i=1}^n$). We claim that both functionals $\mathcal{J}_l^R\big(w_{i+1}(t,\cdot+\widetilde{x}_i(t)\big)$ and $\mathcal{J}_r^R\big(w_{i+1}(t,\cdot+\widetilde{x}_i(t)\big)$ enjoy the almost monotonicity properties \eqref{ineq_J_l}-\eqref{ineq_i_J_r} for $t\geq \tau^R_i\geq t_i^R$ where $\tau_i^R$ is a  sufficiently large parameter to be fixed. In fact, first of all notice that due to \eqref{mod_parameter_bound} and \eqref{inductive_hypothesis} we deduce that for every $\epsilon>0$ there exists $t_\epsilon^i\gg1$ sufficiently large such that for all $t\geq t_\epsilon^i$ we have  
\[
\left\vert \mathcal{J}_l^R\Big(u\big(t,\cdot+\widetilde{x}_i(t)\big)\Big)-\mathcal{J}_l^R\Big(w_{i+1}\big(t,\cdot+\widetilde{x}_i(t)\big)\Big)\right\vert\leq \epsilon,
\]
what proves the assertion for $\mathcal{J}_l^R$. Now, in order to deal with the second case we start by rewriting $\mathcal{J}_r^R$ as \begin{align*}
\mathcal{J}_r^R\big(u(t,\cdot+\widetilde{x}_i(t))\big)&=\int \big(u^2+u_x^2\big)\Psi\big(\cdot-\widetilde{x}_i(t)-R\big)\Big(1-\Psi\big(\cdot-y_{i+1}(t)\big)\Big)
\\ & \quad +\int \big(u^2+u_x^2\big)\Psi\big(\cdot-\widetilde{x}_i(t)-R\big)\Psi\big(\cdot-y_{i+1}(t)\big)=:\mathbf{I}+\mathbf{II}.
\end{align*}
Thus, on the one-hand, by using \eqref{mod_parameter_bound} again we have
\begin{align*}
\mathbf{I}(t)-\int \big(w_{i+1}^2+w_{i+1,x}^2\big)\Psi\big(\cdot-\widetilde{x}_i(t)-R\big)\Big(1-\Psi\big(\cdot-y_{i+1}(t)\big)\Big)\to 0 \ \hbox{ as } \ t\to+\infty,
\end{align*}
while on the other hand, by using the inductive hypothesis \eqref{inductive_hypothesis} together with \eqref{mod_parameter_bound}  we deduce that as $t$ goes to $+\infty$ we have
\begin{align*}
\mathbf{II}(t)-\int \big(w_{i+1}^2+w_{i+1,x}^2\big)\Psi\big(\cdot-\widetilde{x}_i(t)-R\big)\Psi\big(\cdot-y_{i+1}(t)\big)\to \sum_{j=i+1}^nE(\varphi_{c_j^*}).
\end{align*}
Therefore, by gathering both convergences we conclude the claim for $\tau_R^i\gg t_\epsilon^i$ sufficiently large.

\medskip

\textbf{Step 2:} Now we claim that for all $A>0$ the following strong convergence holds: \[
u\big(t,\cdot+\widetilde x_i(t)\big)-\rho_i(t)\varphi-W_{i+1}\big(t,\cdot+\widetilde{x}_i(t)\big)\to 0 \ \hbox{ in }\ H^1\big((-A,+\infty)\big) \ \hbox{ as }\ t\to+\infty.
\]
In fact, it is enough to recall that due to the inductive hypothesis \eqref{inductive_hypothesis} we have that for any $\epsilon>0$ arbitrarily small, there exists $t_\epsilon^i\gg1$ sufficiently large such that for all $t\geq t_\epsilon^i$ we have
\begin{align*}
\int \big(w_{i+1}^2+w_{i+1,x}^2\big)\big(t,x\big)\Psi\big(\cdot-y_{i+1}(t)\big)<\dfrac{\epsilon}{3}.
\end{align*}
Moreover, due to \eqref{mod_parameter_bound}, by making $t_\epsilon^i$ bigger if necessary we can also assume that for any $t\geq t_\epsilon^i$ it holds: \[
\int \big(\varphi^2+\varphi_x^2\big)\big(t,x\big)\Psi\big(\cdot-y_{i+1}(t)\big)<\dfrac{\epsilon}{3}.
\]
Therefore, by gathering the above inequalities together with the almost monotonicity result for $
\mathcal{J}_r^R\big(w_{i+1}(t,\cdot+\widetilde x_i(t))\big)$ with $R=y_{i+1}(t_\epsilon^i)+\widetilde x_i(t_\epsilon^i)$ and the strong convergence result \eqref{local_strong_conv}, we conclude that for all $A>0$ fixed we have \begin{align*}
u\big(t,\cdot+\widetilde x_i(t)\big)-\rho_i(t)\varphi-W_{i+1}\big(t,\cdot+\widetilde{x}_i(t)\big)\to 0 \ \hbox{ in }\ H^1\big((-A,+\infty)\big) \ \hbox{ as }\ t\to+\infty,
\end{align*}
which proves the claim.

\medskip

\textbf{Step 3:} Our aim now is to prove the convergence of the scaling parameter $\rho_i(t)$. In fact, first of all notice that due to \eqref{mod_parameter_bound}, the exponential decay of $\varphi$, $\varphi'$ and $\Psi$ and the latter strong convergence result in $H^1((-A,\infty))$ we deduce that for any $\delta>0$ there exists $R_\delta>1$ and $t_\delta^i>1$ sufficiently larges such that 
\[
\left\vert\int\big(w_{i+1}^2+w_{i+1,x}^2\big)(t,x)\Psi\big(x-\widetilde{x}_i(t)+R_\delta\big)dx-\rho_i(t)E\big(\varphi\big)\right\vert\leq \delta  \ \, \hbox{ for all } \, \ t\geq t_\delta^i.
\]
Then, the almost monotonicity of $\mathcal{J}_l^{R_\delta}(\cdot)$ implies that $\rho_i(t)\to \sqrt{c_i^*}$ as $t\to+\infty$ for some $c_i^*$ close to $c_i$. In fact, from the almost monotonicity of $\mathcal{J}_l^{R_\delta}\big(w_{i+1}(t,\cdot+\widetilde{x}_i(t))\big)$ and the latter inequality it follows that \[
\rho_i(t)E(\varphi)\leq \rho_i(t')E(\varphi)+3\delta, \quad \hbox{for all } \ t\geq t'\geq t_\delta^i.
\]
Since $\delta>0$ is arbritary, this forces $\rho(\cdot)$ to have a limit at $+\infty$, what ends the proof. Notice that, in particular, the following convergence holds\begin{align}\label{convergence_right}
u\big(t,\cdot+\widetilde x_i(t)\big)-\sqrt{c_i^*}\varphi-W_{i+1}\big(t,\cdot+\widetilde{x}_i(t)\big)\to 0 \ \hbox{ in }\ H^1\big((-A,+\infty)\big) \ \hbox{ as }\ t\to+\infty.
\end{align}

\smallskip

\textbf{Step 4:} Now we intend to prove that $\dot{\widetilde{x}_i}(t)\to c_i^*$ as $t\to+\infty$. We point out that the proof is somehow contained in the proof of Lemma \ref{lema_convergence_x_dot_n}. In fact, notice that by differentiating the $i$-th equation in  \eqref{mod_orthogonality} with respect to time and recalling that $\varphi$ satisfies $\varphi-\varphi''=2\delta$ we obtain \[
\int w_{1,t}\omega_{n_0,x}^i=\dot{\widetilde{x}}_i\int w_1(t,x)\omega_{n_0}^i(t,x)dx-2\dot{\widetilde{x}}_i\sqrt{c_i^*}\int w_1(t,x)\rho_{n_0}\big(x-\widetilde{x}_i(t)\big)dx.
\]
On the other hand, notice that due to \eqref{convergence_right}, inequality \eqref{mod_parameter_bound} and the exponential decay of both $\omega_i$ and $\omega_{n_0}^i$ we deduce that $\Vert \omega_j \omega_{n_0}^i\Vert_{L^1}+\Vert \omega_{j,x}\omega_{n_0}^i\Vert_{L^1}\to 0$ as $t\to+\infty$ for $j\neq i$ and 
\[
\Vert w_1^2\omega_{n_0,x}^i\Vert_{L^1}+\Vert w_{1,x}^2\omega_{n_0,x}^i\Vert_{L^1}+\Vert w_1\omega_{n_0}^i\Vert_{L^1}+\Vert w_1\rho_{n_0}(\cdot-\widetilde{x}_i(t))\Vert_{L^1}\to0 \ \hbox{ as }\ t\to+\infty.
\]
Therefore, by taking the $L^2$-inner product from equation \eqref{mod_huge_eq_11} against $\omega_{n_0,x}^i$ and noticing that $\langle \omega_{i,x}(t),\omega_{n_0,x}^i(t)\rangle_{L^2,L^2}\equiv \mathrm{constant}>0$ for all times $t\in\R$, we conclude \[
\dot{\widetilde{x}}_i-c_i^*\to 0 \ \hbox{ as }\ t\to+\infty.
\]

\textbf{Step 5:} Now we intend to conclude the inductive proof, which at the same time proof the convergence result  \eqref{MT_5_conclusions} on the second set in $\mathcal{A}_t$. In fact, first of all let us recall that from the previous steps we know that for any $A>0$ we have that as $t\to+\infty$ the following holds: \[
u\big(t,\cdot+\widetilde{x}_i(t)\big)-\varphi_{c_i^*}-W_{i+1}\big(t,\cdot+\widetilde{x}_i(t)\big)\to0 \, \hbox{ in }\, H^1((-A,\infty)).
\]
Now, let $\eta>0$ arbitrarily small but fixed. Let us consider $R\gg1$ sufficiently large such that \begin{align}\label{smallness_varphi_psi_proooof}
\Vert \varphi\Vert_{H^1\left(\left(-\infty,-\frac{R}{2}\right)\right)}^2<\eta \quad \hbox{and}\quad \Vert \Psi-1\Vert_{L^\infty\left(\left(\frac{R}{2},+\infty\right)\right)}<\eta,
\end{align}
Hence, by the previous convergence results we deduce the existence of a time point $t_0>0$ sufficiently large for which $\widetilde{x}_i(t_0)>R$ and such that for all $t\geq t_0$ we have \[
\left\Vert u\big(t,\cdot+\widetilde{x}_i(t)\big)-\varphi_{c^*_i}-W_{i+1}\big(t,\cdot+\widetilde{x}_i(t)\big)\right\Vert_{H^1\left(\left(-\frac{R}{2},+\infty\right)\right)}<\eta.
\]
Moreover, by making both $R$ and $t_0$ bigger if necessary we can also assume that for all $y\geq R$ and all $t\geq t_0$ it holds \[
\left\vert\sum_{j=i}^nE(\varphi_{c_j^*})-\int \big(W_i^2+W_{i,x}^2\big)\Psi\big(\cdot-\widetilde{x}_i(t)+y\big)dx\right\vert<\eta
\]
On the other hand, by using \eqref{smallness_varphi_psi_proooof}, inequality \eqref{mod_parameter_bound} and the previous inequalities we deduce that for all $y\geq R$ and all $t\geq t_0$ we have
\begin{align}\label{energy_bound}
\left\vert \sum_{j=i}^nE(\varphi_{c_j^*})-\int \left(u(t,\cdot+\widetilde{x}_i(t))W_{i}+u_x(t,\cdot+\widetilde{x}_i(t))W_{i,x}\right)\Psi(\cdot+y)\right\vert\lesssim \eta.
\end{align}
Now, for $j=2,...,n-1$, we consider the following velocities \[
z_1(t):=\tfrac{\beta}{2}t  \quad \hbox{and}\quad z_j(t):=\tfrac{3}{4}x_{j-1}(t)+\tfrac{1}{4}x_j(t).
\]
Now notice that, with this specific choice of velocities $z_i$, the functional $\mathrm{I}_{i,t_0}^{-R}$ defined in Lemma \ref{tech_lem_mon_exp} satisfies the almost monotonicity property. Thus, by using inequality \eqref{AM_right_i_energy}
we obtain that for all $t\geq t_0$ it holds
\begin{align*}
&\int \left(u^2+u_x^2\right)(t,\cdot)\Psi\left(\cdot -z_i^{-R}(t)\right)\leq Ce^{-R/24} +\int \left(u^2+u_x^2\right)(t_0,\cdot)\Psi\left(\cdot -z_i^{-R}(t_0)\right),
\end{align*}
where $z_i^{-R}(t)=\widetilde{x}_i(t_0)-R+z_i(t)-z_i(t_0)$. On the other hand, by straightforward computations we have
\begin{align*}
\int (w_i^2+w_{i,x}^2)(t,\cdot)\Psi\left(\cdot-z_i^{-R}(t)\right)&=\int \left(u^2+u_x^2\right)(t,\cdot)\Psi\left(\cdot -z_i^{-R}(t)\right)
\\ & \quad +\int \left(W_i^2+W_{i,x}^2\right)(t,\cdot)\Psi\left(\cdot -z_i^{-R}(t)\right)
\\ & \quad -2\int \left(uW_i+u_xW_{i,x}\right)(t,\cdot)\Psi\left(\cdot -z_i^{-R}(t)\right)
\\ & =:\mathrm{I}+\mathrm{II}+\mathrm{III}.
\end{align*}
Moreover, notice that due to \eqref{mod_parameter_bound} and the definition of $\{z_j\}_{j=1}^{n-1}$ we deduce that for all $t\geq t_0$ we have \[
\widetilde{x}_i(t)-\widetilde{x}_i(t_0)-z_i(t)+z_i(t_0)+R\geq R.
\]
Hence, by using inequality \eqref{energy_bound} and then the exponential decay of $\varphi$ we get
\begin{align*}
\mathrm{I}+\mathrm{II}+\mathrm{III}&\leq\int \left(u^2+u_x^2\right)(t_0,\cdot)\Psi\left(\cdot -\widetilde{x}_i(t_0)+R\right)+Ce^{-R/24}
\\ & \quad+\int \left(W_i^2+W_{i,x}^2\right)(t_0,\cdot)\Psi\left(\cdot -\widetilde{x}_i(t_0)+R\right)+Ce^{-R/24}
\\ & \quad -2\int \left(uW_i+u_xW_{i,x}\right)(t_0,\cdot)\Psi\left(\cdot -\widetilde{x}_i(t_0)+R\right)+C\eta
\\ & \lesssim \int \big(w_i^2+w_{i,x}^2)(t_0,\cdot)\Psi\left(\cdot-\widetilde{x}_i(t_0)+R\right)+e^{-R/24}+\eta
\\ & \lesssim \eta+e^{-R/24},
\end{align*}
where we have used the exponential decay of $\varphi$ to obtain the latter inequality. Finally, notice that by taking $R\gg1$ sufficiently large and $t_1>t_0$ such that for all $i=2,...,n-1$ and all $t\geq t_1$ it holds \[
z_1^{-R}(t)\leq \tfrac{\beta}{2} \quad \hbox{and} \quad z_i^{-R}(t)\leq y_i(t).
\]
Therefore, recalling that if $i=1$ we defined $y_1(t):=\tfrac{\beta}{2}$ (see the beginning of this section), we conclude that for all $t\geq t_1$ we have
\[
\int (w_{i}^2+w_{i,x}^2)(t,\cdot)\Psi\left(\cdot-y_i(t)\right)\lesssim \eta,
\]
which completes the proof the claim.

\medskip

\textbf{Step 6:}  Finally, it only remains to prove the convergence in $(-\infty,z)$ for any fixed $z\in\R$. This is a consequence of a more general property, noticed by Molinet in \cite{Mo3}, ensuring that all the energy of solutions associated to initial data in $Y_+$ is traveling to the right. In fact, we shall prove the following lemma which immediately conclude the proof of Theorem \ref{MT5}.
\begin{lem}[\cite{Pa}]\label{traveling_energy_lem}
For any $u_0\in Y_+$ and any $z\in\R$, the solution $u\in C(\R,H^1(\R))$ to equation \eqref{nov_eq_2} associated $u_0$ satisfies \begin{align*}
\lim_{t\to+\infty}\Vert u(t)\Vert_{H^1((-\infty,z))}=0.
\end{align*}
\end{lem}

\begin{proof}[Proof of Lemma \ref{traveling_energy_lem}]
This lemma has already been proved for the Novikov equation in \cite{Pa}. However, for the sake of completeness we prove it again. First of all notice that, for $\Psi$ defined in \eqref{psi_def_2}, for any time $t\in\R$ fixed the map \[
z\mapsto\int \big(u^2+u_x^2\big)(t,x)\Psi(\cdot-z)dx,
\]
defines a decreasing continuous bijection from $\R$ into $(0,\Vert u_0\Vert_{H^1}^2)$. Hence, by setting any $0<\gamma<\Vert u_0\Vert_{H^1}^2$, we deduce that the map $x_\gamma:\R\to\R$ defined by the equation \begin{align}\label{def_xgamma}
\int \big(u^2+u_x^2\big)(t,x)\Psi(\cdot-x_\gamma(t))dx=\gamma,
\end{align}
is well-defined. Moreover, since $u\in C(\R,H^1(\R))$ we deduce that $x_\gamma$ is a continuous function. Now, notice that in order to conclude the proof of the lemma it is enough to show that for any $\gamma\in(0,\Vert u_0\Vert_{H^1}^2)$ we have \begin{align}\label{lim_final}
\lim_{t\to+\infty} x_\gamma(t)=+\infty.
\end{align}
For the sake of readability we split the proof of the latter property in two steps.

\medskip

\textbf{Step 1:} First we claim that for any $\Delta>0$ and any $t\in\R$ we have \begin{align}\label{claim_step}
x_\gamma(t+\Delta)-x_\gamma(t)\geq \dfrac{2}{5}\int_t^{t+\Delta}\int u^2(t,x)\Psi'(\cdot-x_\gamma(t))dx >0.
\end{align}
In fact, notice that by continuity with respect to the initial data it is enough to prove the claim for solutions $u\in C^\infty(\R,H^\infty(\R))\cap L^\infty(\R,H^1(\R))$. On the other hand, as an application of the Implicit Function Theorem we deduce that $x_\gamma(t)$ is of class $C^1$. In fact, let us define the functional \[
\psi(v,z):=\int \big(v^2+v_x^2\big)\Psi(\cdot-z)dx.
\]
Notice that $\psi$ clearly defines a $C^1$ function on $H^1(\R)\times \R$. Moreover, notice that since any function $v\in Y_+\setminus\{0\}$ cannot vanish at any point $x\in\R$, we deduce that for any function $v\in H^\infty\cap Y_+$ and any $z\in\R$ we have \[
\dfrac{\partial\psi}{\partial z}=\int \big(v^2+v_x^2\big)\Psi'(\cdot-z)>0.
\]
Thus, recalling equation \eqref{dt_I_J_i} from the proof of Lemma \ref{tech_lem_mon_exp}, we obtain 
\begin{align*}
\dot{x}_\gamma\int \big(u^2+u_x^2\big)\Psi'(\cdot-x_\gamma)&=\int u^2u_x^2\Psi'+\int \{p*(3uu_x^2+2u^3)\}u\Psi'+\int \{p_x*u_x^3\}u\Psi'.
\end{align*}
Now, due to the fact that $\vert v_x\vert\leq v$ for any $v\in Y_+$ we deduce $p*uu_x^2+p_x*u_x^3\geq 0$. On the other hand, since $u(t)$ is positive, from Lemma $3.7$ in \cite{Pa} we deduce
\[
p*(3uu_x^2+5u^3)\geq 2u^3 \ \hbox{ in particular } \ p*(2uu_x^2+2u^3)\geq \tfrac{4}{5}u^3.
\]
Hence, by using again that $\vert v_x\vert\leq v$ for any $v\in Y_+$ and the previous inequalities we get
\begin{align*}
2\dot{x}_\gamma\int u^2\Psi'(\cdot-x_\gamma)&\geq \int u^2u_x^2\Psi'+\dfrac{4}{5}\int u^4\Psi'.
\end{align*}
Therefore, due to the non-negativity of $\Psi'$ together with the fact that $\Vert \Psi'\Vert_{L^1}=1$ and by using H\"older's inequality we get 
\begin{align*}
\dot{x}_\gamma(t)\geq \dfrac{2}{5}\int u^2\Psi'(\cdot-x_\gamma(t))dx.
\end{align*}
Integrating in time between $t$ and $t+\Delta$ we conclude the claim.

\medskip

\textbf{Step 2:} Now we intend to conclude the proof of \eqref{lim_final}. First of all notice that from the claim of the previous step we obtain, in particular, that $x_\gamma(\cdot)$ is increasing and hence it has a limit $x_\gamma^\infty\in \R\cup\{+\infty\}$, i.e. \begin{align*}
\lim_{t\to+\infty}x_\gamma(t)=x_\gamma^\infty.
\end{align*}
Thus, the proof of \eqref{lim_final} is equivalent to prove that $x_\gamma^\infty=+\infty$. In fact, let us proceed by contradiction, i.e. let us suppose that $x_\gamma^\infty\in\R$. Then, notice that the latter hypothesis together with inequality \eqref{def_xgamma} and the fact that $\vert u_x\vert\leq u\leq \Vert u_0\Vert_{H^1}$ for all $(t,x)\in\R^2$ implies 
\begin{align}\label{contrad_2}
\lim_{t\to+\infty}\int\big(u^2+u_x^2\big)(t,x)\Psi(\cdot-x_\gamma(t))=\lim_{t\to+\infty}\int \big(u^2+u_x^2\big)(t,x)\Psi(\cdot-x_\gamma^\infty)=\gamma.
\end{align}
On the other hand, by taking $\Delta=1$, from \eqref{claim_step} and the convergence of $x_\gamma(t)$ we deduce 
\[
\lim_{t\to+\infty}\int_{t}^{t+1}\int u^2\Psi'(\cdot-x_\gamma(t))=\lim_{t\to+\infty}\int_t^{t+1}\int u^2\Psi'(\cdot-x_\gamma^\infty)=0.
\]
Notice that the latter equality together with the fact that $\vert v_x\vert\leq v$ for all $v\in Y_+$ implies, in particular, that there exists a sequence of times $t_n\to+\infty$ such that for any compact set $K\subset\R$ the following holds:
\begin{align}\label{zero_compact}
\lim_{n\to+\infty}\Vert u(t_n)\Vert_{L^\infty(K)}=0.
\end{align}
Now we choose any $\gamma<\gamma'<\Vert u_0\Vert_{H^1}$, arbitrary but fixed. Then, we consider the compact set \[
K:=[x^\infty_\gamma-M,x_\gamma^\infty+M],
\]
with $M\gg1$ sufficiently large such that $x_\gamma^\infty-M<x_{\gamma'}(0)$. Thus, by using \eqref{zero_compact}, the monotonicity of $t\mapsto x_{\gamma'}(t)$ and recalling that $x_{\gamma'}(0)<x_{\gamma}(0)$ we conclude \[
\lim_{n\to+\infty}\int \big(u^2+u_x^2\big)(t_n,x)\Psi(\cdot-x_{\gamma}^\infty)=\gamma'.
\]
However, this contradicts hypothesis \eqref{contrad_2} what ends the proof of the lemma.
\end{proof}

\medskip

\section{Appendix}

\subsection{Proof of Lemma \ref{mod_lemma}}\label{mod_appendix}

Let $\vec{z}=(z_1,...,z_n)\in\R^n$ be fixed and satisfying $z_i-z_{i-1}>L$. Consider the functionals given by the orthogonality conditions we are looking for, i.e., for each $i=1,...,n$ consider the functional given by 
\begin{align*}
Y_i(y_1,...,y_n,u):=\int \big(u-R_{\vec{z}-\vec{x}}\big)\partial_x(\rho_m*\varphi_{c_i})(\cdot-z_i-y_i)dx.
\end{align*}
Notice that each $
Y_i:\R^n\times H^1\to \R$ defines a $\mathcal{C}^1$ functional in a neighborhood of $(0,...,0,R_{\vec{z}})$. Moreover, for any $\vec{z}\in\R^n$ we have $Y_i(0,...,0,R_{\vec z})=0$. For the sake of simplicity, from now on we denote by $Y$ the functional given by \[
Y(y_1,...,y_n,u):=\big(Y_1(y_1,...,y_n,u),...,Y_n(y_1,...,y_n,u)\big).
\]
Now, notice that for each $i=1,...,n$ we have \[
\dfrac{\partial Y_i}{\partial y_i}=\int\Big(u_x-\sum_{j\neq i}^n\partial_x\varphi_{c_j}(\cdot-z_j-y_j)\Big)\partial_x(\rho_m*\varphi_{c_i})(\cdot-z_i-y_i)dx,
\]
and for each $j=1,...,n$ with $j\neq i$ we have \[
\dfrac{\partial Y_i}{\partial y_j}=\int \partial_{x}\varphi_{c_j}(\cdot-z_j-y_j)\partial_x(\rho_m*\varphi_{c_i})(\cdot-z_i-y_i)dx.
\]
In particular, notice that there exists a constant $\textbf{C}>0$ depending only on $m\in\N$ such that for all $i=1,...,n$ and all $\vec{z}\in\R^n$ we have
\[
\dfrac{\partial Y_i}{\partial y_i}(0,...,0,R_{\vec{z}})=\int  \varphi_{c_i}'(\cdot-z_i)(\rho_m'*\varphi_{c_i})(\cdot-z_i)= \textbf{C}c_i\geq \textbf{C}c_1.
\]
On the other hand, by using the exponential decay of $\varphi$ and due to the fact that $z_i-z_j>L$ whenever $i\neq j$, we infer that for $L_0$ large enough we have
\begin{align*}
\left\vert\dfrac{\partial Y_i}{\partial y_j}(0,...,0,R_{\vec{z}})\right\vert&=\left\vert\int\varphi_{c_j}(\cdot-z_j)(\rho_m''*\varphi_{c_i})(\cdot-z_i)\right\vert
\\ &\leq \left \vert \int \varphi_{c_j}(\cdot-z_j)(\rho_m*\varphi_{c_i})(\cdot-z_i)\right\vert+\left\vert\int\varphi_{c_j}(\cdot-z_j)\rho_m(\cdot-z_i)\right\vert
\\ & =O\left(e^{-L/4}\right).
\end{align*}
Hence, for $L\gg1$ large enough we deduce that $D_{\vec{x}}Y(0,...,0,R_{\vec{z}})=D+P$, where $D$ is an invertible diagonal matrix and \[
\Vert D^{-1}\Vert\leq \dfrac{1}{\textbf{C}c_1} \quad\hbox{and} \quad \Vert P\Vert\leq O\left(e^{-L/4}\right).
\]
Thus, there exists $L_0>1$ such that for all $L>L_0$ and all $\vec{z}\in\R^n$ satisfying $z_i-z_{i-1}>L$, the jacobian $D_{\vec{x}}Y(0,...,0,R_{\vec{z}})$ defines an invertible matrix. Hence, by using the Implicit Function Theorem we infer the existence of positive constants $\delta>0$, $C_0>0$ and $C^1$ functions $(y_1,...,y_n)$ defined in a $H^1$-neighborhood of $R_{\vec{z}}$ with values in a neighborhood of zero, i.e. \[
y_1,...,y_n:B_{H^1}(R_{\vec{z}},\delta)\to B_\R(0,C_0\delta),
\]
which are uniquely determined by the equation \[
Y\big(y_1(u),...,y_n(u),u\big)=0 \quad \hbox{for any } \ u\in B_{H^1}\big(R_{\vec{z}},\delta\big).
\]
In particular, there exists a constant $K_0>0$ such that if $u\in B_{H^1}\left(R_{\vec{z}},\delta_*\right)$ for some $0<\delta_*\leq\delta$, then \begin{align}\label{small_neigh}
\sum_{i=1}^n\vert y_i(u)\vert\leq K_0\delta_*.
\end{align}
It is worth noticing that $\delta$ and $K_0$ only depend on $c_1$ and $L_0$ but not on the point $\vec{z}\in\R$. Thus, for $u\in B_{H^1}(R_{\vec{z}},\delta)$ we can set $\widetilde{x}_i(u):=z_i+y_i(u)$. Hence, assuming that $\delta\leq \tfrac{L_0}{8K_0}$, we infer that $\widetilde{x}_1,...,\widetilde{x}_n$ are $C^1$ functions on $B_{H^1}(R_{\vec{z}},\delta_*)$ and satisfy \begin{align}\label{distances_ap}
\widetilde{x}_i(u)-\widetilde{x}_{i-1}(u)=z_i-z_{i-1}+y_i(u)-y_{i-1}(u)>\dfrac{L}{2}-2K_0\delta_*>\dfrac{L}{4}.
\end{align}
Now we intend to define the modulation of $u$. In fact, let us consider $\alpha_0<\tfrac{1}{2}\delta$ to be chosen later. Then, for all $L\geq L_0$ and any $0<\alpha<\alpha_0$, we define the modulation of $u$ in the following way: We cover the trajectory of $u$ by a finite number of open balls by: \[
\left\{u(t):\ t\in[0,t_0]\right\}\subset\bigcup_{k=1,...,N}B_{H^1}\left(R_{\vec{z}_k},2\alpha\right)
\] 
It is important to notice that, since $0<\alpha<\alpha_0<\tfrac{1}{2}\delta$, the functions $\widetilde{x}_i(u)$ are uniquely determined for \[
u\in B\left(R_{\vec{z}_k},2\alpha\right)\cap B\left(R_{\vec{z}_{k'}},2\alpha\right).
\] 
Therefore, we can define the functions $t\mapsto \widetilde{x}_i(t)$ on $[0,t_0]$ by settin $\widetilde{x}_i(t):=\widetilde{x}_i(u(t))$. Thus, by construction \begin{align}\label{ort_proof}
\int\Big(u(t,\cdot)-\sum_{j=1}^n\varphi_{c_j}\big(\cdot-\widetilde{x}_j(t)\big)\Big)\partial_x(\rho_m*\varphi_i)\big(\cdot-\widetilde{x}_i(t)\big)dx=0.
\end{align}
On the other hand, where $k$ is such that at time $t$ we have $u(t)\in B(R_{\vec{z}_k},2\alpha)$, by direct computation we get \begin{align*}
\left\Vert u(t)-\sum_{i=1}^n\varphi_{c_i}\big(\cdot-\widetilde{x}_i(t)\big)\right\Vert_{H^1}&\leq \left\Vert u(t)-\sum_{i=1}^n\varphi_{c_i}\big(\cdot-z_{i,k}\big)\right\Vert_{H^1}
\\ & \quad +\left\Vert\sum_{i=1}^n\big(\varphi_{c_i}\big(\cdot-\widetilde{x}_i(t)\big)-\varphi_{c_i}\big(\cdot-z_{i,k}\big)\big)\right\Vert_{H^1}
\\ & \leq 2\alpha+\Big(2\sum_{i=1}^nE(\varphi_{c_i})-2\sum_{i=1}^n\int \varphi_{c_i}(\cdot-\widetilde{x}_i(t))\varphi_{c_i}(\cdot-z_{i,k})
\\ & \quad -2\sum_{i=1}^n\int \varphi_{c_i}'(\cdot-\widetilde{x}_i(t))\varphi_{c_i}'(\cdot-z_{i,k})\Big)^{1/2} =:2\alpha+\mathrm{II}
\end{align*}
by using \eqref{small_neigh} , $E(\varphi_{c_i})=2c_i$ by integration by parts and recalling that $\varphi''=\varphi-2\delta$ we obtain \begin{align*}
\mathrm{II}&=2\sum_{i=1}^n\big(c_i-\varphi_{c_i}(z_{i,k}-\widetilde{x}_i(t)\big)^{1/2}= 2\sum_{i=1}^n\sqrt{c_i}\left(1-e^{-\vert z_{i,k}-\widetilde{x}_i(t)\vert}\right)^{1/2}
\\ & \leq 2\sum_{i=1}^n\sqrt{c_i}\left(1-e^{-2K_0\alpha}\right)^{1/2}\leq 2\sqrt{2K_0\alpha}\sum_{i=1}^n\sqrt{c_i}=O\left(\sqrt{\alpha}\right)
\end{align*}
Therefore, for $\alpha\ll1$ small enough we conclude that for all $t\in[0,t_0]$ it holds \begin{align}\label{small_mod_ap_proof}
\left\Vert u(t)-\sum_{i=1}^n\varphi_{c_i}\big(\cdot-\widetilde{x}_i(t)\big)\right\Vert_{H^1}=O\left(\sqrt{\alpha}\right).
\end{align}
Now, we split the prove of the remaining inequalities into four steps.

\medskip

\textbf{Step 1:} Now we intend to prove the first inequality in \eqref{mod_parameter_bound}. For the sake of simplicity let us start by defining some auxiliary variables: For each $i=1,...,n$ we define the functions $v,w_i,w_m^i$ as \[
v(t):=u(t)-\sum_{i=1}^n\varphi_{c_i}(\cdot-\widetilde{x}_i(t)),\quad w_i:=\varphi_{c_i}\big(\cdot-\widetilde{x}_i(t)\big) \ \hbox{ and } \ w_{m}^i(t):=(\rho_m*\varphi_{c_i})\big(\cdot-\widetilde{x}_i(t)\big).
\]
Then, by differentiating \eqref{ort_proof} and recalling that $\varphi-\varphi''=2\delta$ we obtain \begin{align*}
\left\vert \int v_t(t,x)w_{m,x}^i(t,x)dx\right\vert&=\left\vert\dot{\widetilde{x}}_i(t)\langle w_{m,xx}^i,v\rangle_{H^{-1},H^1}\right\vert\leq \left\vert\dot{\widetilde{x}}_i(t)\right\vert O\big(\Vert v(t)\Vert_{H^1}\big)
\\ & \leq \left\vert\dot{\widetilde{x}}_i(t)-c_i\right\vert O\big(\Vert v(t)\Vert_{H^1}\big)+O\big(\Vert v(t)\Vert_{H^1}\big).
\end{align*}
On the other hand, by using that $\varphi$ solves \eqref{nov_eq_2} we infer that each $w_i$ satisfies the following equation:
\[
w_{i,t}+\left(\dot{\widetilde{x}}_i-c_i\right)w_{i,x}+w_i^2w_{i,x}=p_x\left(w_i^3+\dfrac{3}{2}w_iw_{i,x}^2\right)-\dfrac{1}{2}p*w_{i,x}^3
\]
Thus, by using that $u(t)$ also solves \eqref{nov_eq_2}, replacing $u=v+w_1+...+w_n$ and then using the equation satisfied by each $w_i$ we get \begin{align}\label{mod_huge_eq_ap}
v_{t}-\sum_{j=1}^n(\dot{\widetilde{x}}_j-c_j)w_{j,x}&=-\left(v+\sum_{j=1}^nw_j\right)^2v_{x}-v^2\sum_{j=1}^nw_{j,x}-2v\sum_{j,k=1}^nw_jw_{k,x}\nonumber
\\ & \quad -\sum_{\substack{j,k,\ell=1 \\ (k,\ell)\neq(j,j)}}^nw_jw_kw_{\ell,x}-p_x*v^3-3\sum_{j=1}^np_x*(v^2w_j)\nonumber
\\ &  \quad  -3\sum_{j,k=1}^np_x*vw_jw_k-\sum_{\substack{j,k,\ell=1 \\ (k,\ell)\neq(j,j)}}^np_x*w_jw_kw_\ell \nonumber
\\ & \quad -\dfrac{3}{2}p_xvv_{x}^2-3\sum_{j=1}^np*vv_{x}w_{j,x}-\dfrac{3}{2}\sum_{j,k=1}^np*vw_{j,x}w_{k,x}\nonumber
\\ & \quad -\dfrac{3}{2}\sum_{j=1}^np*v_x^2w_j-3\sum_{j,k=1}^np*v_xw_jw_{k,x}\nonumber
\\ & \quad -\dfrac{3}{2}\sum_{\substack{j,k,\ell=1 \\ (k,\ell)\neq(j,j)}}^np_x*w_jw_{k,x}w_{\ell,x}-\dfrac{1}{2}p*v_x^3-\dfrac{3}{2}\sum_{j=1}^np*v_x^2w_{j,x}\nonumber
\\ & \quad -\dfrac{3}{2}\sum_{j,k=1}^n p*v_xw_{j,x}w_{k,x}-\dfrac{1}{2}\sum_{\substack{j,k,\ell=1 \\ (k,\ell)\neq(j,j)}}^n p*w_{j,x}w_{k,x}w_{\ell,x}.
\end{align}
On the other hand, notice that due to \eqref{small_mod_ap_proof}, inequality \eqref{distances_ap} and the exponential decay of both $w_i$ and $w_{n_0}^i$ we deduce that $\Vert w_j w_{m}^i\Vert_{L^1}+\Vert w_{j,x}w_{m}^i\Vert_{L^1}=O\left(\exp(-L/4)\right)$ for $j\neq i$ and 
\[
\Vert v^2w_{m,x}^i\Vert_{L^1}+\Vert v_{x}^2w_{m,x}^i\Vert_{L^1}+\Vert vw_{m}^i\Vert_{L^1}+\Vert v\rho_{m}(\cdot-\widetilde{x}_i(t))\Vert_{L^1}=O\left(\sqrt{\alpha}\right)+O\left(e^{-L/4}\right).
\]
Hence, by taking the $L^2$-inner product from equation \eqref{mod_huge_eq_ap} against $w_{m,x}^i$ and noticing that there exists a constant $\textbf{a}>0$ such that $\langle w_{i,x}(t),w_{m,x}^i(t)\rangle_{L^2,L^2}\equiv \textbf{a}c_i>0$ for all times $t\in[0,t_0]$, we obtain \[
\left\vert\dot{\widetilde{x}}_i-c_i\right\vert\Big(\textbf{a}c_i+O\left(\sqrt{\alpha}\right)\Big)=O\left(\sqrt{\alpha}\right)+O\left(e^{-L/4}\right).
\]
Therefore, by taking $0<\alpha_0\ll1$ small enough and $L_0\gg1$ sufficiently large we conclude the first inequality in \eqref{mod_parameter_bound}.

\medskip

\textbf{Step 2:} Now we intend to prove the second inequality in \eqref{mod_parameter_bound}. In fact,  it is enough to notice that from what we proved in the last step, by using \eqref{initial_cond_hyp_train} and \eqref{small_neigh}, after integration in time we obtain \[
\widetilde{x}_i(t)-\widetilde{x}_{i-1}(t)\geq L-2K_0\alpha_0+\dfrac{c_i-c_{i-1}}{2}t\geq \dfrac{3}{4}L+\dfrac{c_i-c_{i-1}}{2}t,
\]
what finish the proof of \eqref{mod_parameter_bound}.

\medskip

\textbf{Step 3:} In this step we are devoted to prove last part of the statement. In fact, notice that by using \eqref{small_mod_ap_proof} together with Sobolev's embedding we infer that for any time $t\in[0,t_0]$ we have \[
u(t,x)=R_{\widetilde{x}(t)}(x)+O\left(\sqrt{\alpha}\right).
\]
Now, on the one-hand notice that by applying the previous formula with $x(t):=\max_{J_i}u(t)$ and by using the second inequality in \eqref{mod_parameter_bound} we obtain \[
u(t,x_i)=\sqrt{c_i}+O\left(\sqrt{\alpha}\right)+O\left(e^{-L/4}\right)\geq \tfrac{2}{3}\sqrt{c_i}.
\]
On the other hand, notice that for any $x\in J_i\setminus[\widetilde{x}_i(t)-\tfrac{1}{12}L,\widetilde{x}_i(t)+\tfrac{1}{12}L]$ we have \[
u(t,x)\leq \sqrt{c_i}e^{-L/12}+O\left(\sqrt{\alpha}\right)+O\left(e^{-L/4}\right)\leq \tfrac{1}{2}\sqrt{c_i}.
\]
Therefore, the previous two inequalities ensures that $x_i(t)\in[\widetilde{x}_i(t)-\tfrac{1}{12}L,\widetilde{x}_i(t)+\tfrac{1}{12}L]$, what concludes the proof of the lemma.

\medskip

\textbf{Step 4:} Finally, it only remains to prove \eqref{orth_cond_def} for $n_0\in\N$ large enough. In fact, it is enough to notice that  \[
\int \varphi'(x)\varphi(x-y)dx=(1-y)e^{-y}.
\]
Thus, for $n_0\in\N$ large enough we have
\[
\dfrac{d}{dy}\int \varphi(\rho_{n_0}*\varphi)'(\cdot-y)=\int \varphi'(\rho_{n_0}*\varphi')(\cdot-y)\geq \dfrac{1}{4}e^{-1/2} \quad \hbox{on} \quad \left[-\tfrac{1}{2},\tfrac{1}{2}\right].
\]
Therefore, the mapping $y\mapsto \int_\R\varphi(\rho_{n_0}*\varphi)'(\cdot-y)$ is increasing on $[-\tfrac{1}{2},\tfrac{1}{2}]$, and hence there exists $n_0\in\N$ satisfying \eqref{orth_cond_def}. Then, we conclude the proof by choosing $m=n_0$. \qed

\medskip

\subsection{Proof of Lemma \ref{AM_orb_train}}\label{AM_orb_train_appendix}
The following computations can be rigorized by standard approximation and density arguments by considering, for instance, the convolution of $u_0$ with the mollifiers family $\rho_n$ defined in \eqref{def_rho} and by using the second statement in Theorem \ref{theorem_lwp}. We refer to \cite{EM1} for a complete justification of this argument. 

\medskip

Now, our aim is to prove inequality \eqref{AM_energy_orbital} by integrating its time derivative. Hence, by taking the time derivative directly from the definition of $\mathcal{I}_{i,K}(t)$ we obtain
\begin{align}
\dfrac{d}{dt}\mathcal{I}_{i,K}(t)&=2\int \big(uu_t+u_xu_{xt}\big)\Psi_{i,K}-\dot y_i(t)\int (u^2(t)+u_x^2(t)\big)\Psi_{i,K}'\nonumber
\\ & =:\mathrm{J}-\dot y_i(t)\int (u^2(t)+u_x^2(t)\big)\Psi_{i,K}'.\label{dt_I_thm}
\end{align}
By using both equations \eqref{novikov_eq} and \eqref{nov_eq_2} and by integrating by parts we get
\begin{align*}
\mathrm{J}&=2\int \big(u_t-u_{txx}\big)u\Psi_{i,K}-2\int uu_{tx}\Psi_{i,K}'
\\ & =2\int \big(3uu_xu_{xx}+u^2u_{xxx}-4u^2u_x\big)u\Psi_{i,K}
\\ & \quad +2\int (u^2u_{xx}+2uu_x^2+p_x*(3uu_xu_{xx}+2u_x^3+3u^2u_x))u\Psi_{i,K}'
\\ & =4\int u^2u_x^2\Psi_{i,K}'+2\int u^4\Psi_{i,K}'+2\int \{p_x*(3uu_xu_{xx}+2u_x^3+3u^2u_x)\}u\Psi_{i,K}'
\end{align*}
On the other hand, recalling that for any $L^2$ function $f:\R\to\R$ we have $p*f_x=p_x*f$, and by using that $p$ is the fundamental solution of $(1-\partial_x^2)$, we obtain
\begin{align*}
2p_x*(3uu_xu_{xx}+2u_x^3+3u^2u_x)=-2u^3-3uu_x^2+3p*uu_x^2+2p*u^3+p_x*u_x^3.
\end{align*}
Thus, by plugging this into \eqref{dt_I_thm} we get \begin{align}
\dfrac{d}{dt}\mathcal{I}_{i,K}(t)&=-\dot y_i(t)\int \big(u^2+u_x^2\big)\Psi_{i,K}'+\int u^2u_x^2\Psi_{i,K}'\nonumber
\\ & \qquad +\int \{p*(3uu_x^2+2u^3)\}u\Psi_{i,K}'+\int \{p_x*u_x^3\}u\Psi_{i,K}'\nonumber
\\ & =-\dot y_i(t)\int \big(u^2+u_x^2\big)\Psi_{i,K}'+\mathrm{J}_1+\mathrm{J}_2+\mathrm{J}_3.\label{dt_I_J_i} 
\end{align}
In order to bound $\mathrm{J}_i$, for $i=1,2,3$, we split $\R$ into two complementary regions related to the size of $u(t)$. In fact, for $i=2,...,n$ let us consider the family of time-dependent intervals $D_i$ \[
D_i(t):=\big[\widetilde{x}_{i-1}(t)+\tfrac{L}{4},\,\widetilde x_i(t)-\tfrac{L}{4}\big].
\]
Hence, with these definitions, by splitting the space into $D_i$ and $D_i^c$ we can rewrite $\mathrm{J}_1$ as \[
\mathrm{J}_1=\int_{D_i} u^2u_x^2\Psi_{i,K}'+\int_{D_i^c} u^2u_x^2\Psi_{i,K}'=:\mathrm{J}_1^1+\mathrm{J}_1^2.
\]
Now notice that, on the one-hand, by using \eqref{mod_bound} we deduce that for all $t\in[0,t^*]$ we have
\begin{align*}
\Vert u(t)\Vert_{L^\infty(D_i)}&\leq \sum_{i=1}^n\left\Vert \varphi_{c_i}\big(\cdot-\widetilde x_i(t)\big)\right\Vert_{L^\infty(D_i)}+\left\Vert u(t)-\sum_{i=1}^n\varphi_{c_i}\big(\cdot-\widetilde x_i(t)\big)\right\Vert_{L^\infty(D_i)}
\\ &\leq Ce^{-L/8}+O\big(\sqrt\alpha\big).
\end{align*}
Thus, for $\alpha>0$ small enough we can absorb $\mathrm{J}_1^1$ by the first integral term in \eqref{dt_I_J_i}. Now, on the other hand, by using the definition of the family $y_i$ in \eqref{def_y_i_intervals} and by using inequality \eqref{mod_parameter_bound}, we deduce
\[
\hbox{for any }\,x\in D_i^c \,\hbox{ we have } \ \vert x-y_i(t)\vert\geq \tfrac{1}{2}\big(\widetilde x_i(t)-\widetilde x_{i-1}(t)\big)-\tfrac{L}{4}\geq \tfrac{1}{2}(c_i-c_{i-1})t+\tfrac{L}{8}.
\]
Therefore, by using the exponential decay of $\Psi_{i,K}'$, we deduce that for all $t\in[0,t^*]$ we have \begin{align}\label{J_one_two_ap}
\mathrm{J}_1^2=\int_{D_i^c}u^2u_x^2\Psi'_{i,K}\leq \dfrac{C}{K}\Vert u_0\Vert_{H^1}^4e^{-\frac{1}{K}(\sigma_0 t+L/8)}.
\end{align}
Now in order to deal with $\mathrm{J}_2$ we proceed in a similar fashion. First, we split $\mathrm{J}_2$ into two different integrals by using the definition of $D_i$. In concrete, we define
\[
\mathrm{J}_2=\int_{D_i}\{p*(3uu_x^2+2u^3)\}u\Psi_{i,K}'+\int_{D_i^c}\{p*(3uu_x^2+2u^3)\}u\Psi_{i,K}'=:\mathrm{J}_2^1+\mathrm{J}_2^2.
\]
Now notice that in order to follow the previous procedure we need to deal with the self-adjoint operator $(p*\cdot)$. However, it is enough to notice that for $K>4$, by using the definition of $\Psi_{i,K}$ in \eqref{def_Psi_i_K} we immediately obtain
\[
(1-\partial_x^2)\Psi_{i,K}'\geq \left(1-\dfrac{10}{K^2}\right)\Psi_{i,K}', \ \,\hbox{ and hence } \ \, (1-\partial_x^2)^{-1}\Psi_{i,K}'\leq \left(1-\dfrac{10}{K^2}\right)^{-1}\Psi_{i,K}'.
\]
Thus, by using the previous estimate and proceeding in a similar fashion as before we obtain
\[
\mathrm{J}_2^1\lesssim \Vert u(t)\Vert_{L^\infty(D_i)}^2\int_{D_i}\big(u^2+u_x^2\big)(1-\partial_x^2)^{-1}\Psi_{i,K}'\lesssim \Vert u(t)\Vert_{L^\infty(D_i)}^2\int\big(u^2+u_x^2\big)\Psi_{i,K}'.
\]
Hence, for $\alpha>0$ small enough this term can be absorb by the first integral term in \eqref{dt_I_J_i}. Finally, by using the exponential decay of $\Psi_{i,K}'$ and the definition of $D_i$ we get
\[
\mathrm{J}_2^2=\int_{D_i^c}\{p*(3uu_x^2+2u^3)\}u\Psi_{i,K}'\leq \dfrac{C}{K}\Vert u_0\Vert_{H^1}^4e^{-\frac{1}{K}(\sigma_0t+L/8)}.
\]
The remaining term can be bound in exactly the same fashion. Therefore, gathering all the previous estimates we get \begin{align*}
\dfrac{d}{dt}\mathcal{I}_{i,K}(t)\leq -\dfrac{c_1}{4}\int \big(u^2+u_x^2\big)\Psi_{i,K}'+\dfrac{C}{K}\Vert u_0\Vert_{H^1}^4e^{-\frac{1}{K}(\sigma_0t+L/8)}.
\end{align*}
Integrating the previous inequality between $0$ and $t$ we conclude. The proof is complete. \qed

\medskip

\subsection{Proof of Lemma \ref{tech_lem_mon_exp}}\label{sec_ap_tech_lem}
For the sake of simplicity, we split the proof into three steps. The first of them is devoted to proof inequality \eqref{AM_right_n_energy} while the other two aim to prove \eqref{AM_right_i_energy}.

\medskip

\textbf{Step 1:}  First of all notice that by considering $R_0>0$ sufficiently large so that \begin{align}\label{R_condition}
nc_ne^{-R_0}<\dfrac{\sigma}{2^{18}},
\end{align}
and combining this inequality together with \eqref{orb_concl} and the definitions in \eqref{parameters}, we immediately deduce that \eqref{condition_tail} is satisfied. Now, we set $t_R^n$ to be \[
t_R^n:=\max\Big\{\{0\}\cup\big\{t\in t\geq 0: \ \widetilde{x}_n(t)-\widetilde{x}_{n-1}(t)=2R\big\}\Big\}.
\]
On the other hand, recall that from the proof of Lemma \ref{AM_orb_train} (c.f. \eqref{dt_I_J_i}) we have \begin{align}\label{derivative_ap}
\dfrac{d}{dt}\mathrm{I}_{n,t_0^n}^R=-\dot{z}_n^R(t)\int \big(u^2+u_x^2\big)\Psi'(\cdot-z_n^R(t))dx+\mathrm{J}_1+\mathrm{J}_2+\mathrm{J}_3.
\end{align}
Thus, by splitting the space into two regions: \[
\R=\big(-\infty,\widetilde{x}_n(t)+R_0\big]\cup\big[\widetilde{x}_n(t)+R_0,+\infty)=:D_1\cup D_2,
\]
we deduce that for any $x\leq \widetilde{x}_n(t)+R_0$ and any $t\leq t_0^n$ we have \[
x-z_n^R(t)\leq R_0-R-\delta_n(t_0^n-t), \ \, \hbox{ and hence } \ \ \mathrm{J}_1^1\leq \Vert u_0\Vert_{H^1}^4e^{\frac{1}{6}R_0-\frac{1}{6}R-\frac{1}{6}\delta_n(t_0^n-t)},
\]
where $\mathrm{J}_1^1$ is the portion of $\mathrm{J}_1$ associated to $D_1$. On the other hand, by using \eqref{condition_tail} and proceeding in the same fashion as in the proof of Lemma \ref{AM_orb_train} we deduce that $\mathrm{J}_1^2$ can be absorbed by the first integral term in \eqref{derivative_ap}. The remaining terms can be treated in exactly the same fashion. Therefore, by integration in time we conclude the first inequality in \eqref{AM_right_n_energy}.

\medskip

Now we intend to prove the second inequality in \eqref{AM_right_n_energy}. In fact, by using the first inequality in \eqref{mod_parameter_bound} we deduce that for all $R\geq R_0$ we have \[
\vert c_n-\dot{\widetilde{x}}_n(t)\vert+\vert c_{n-1}-\dot{\widetilde{x}}_{n-1}(t)\vert\leq \dfrac{1}{12}(c_n-c_{n-1}), \ \ \hbox{ for all } t\geq 0.
\]
Moreover, defining the time-dependent interval $\Pi_n(t):=(\tfrac{5}{6}\widetilde{x}_{n-1}(t)+\tfrac{1}{6}\widetilde{x}_n(t),\widetilde{x}_n(t)-R_0)$ we deduce from this choice of parameters that  \[
\Vert u(t)\Vert_{L^\infty(\Pi_n(t))}\leq \dfrac{(1-\delta_i)c_n}{\mathbf{b}}, \ \ \hbox{ for all } \ t\geq t_R^n.
\]
Thus, gathering the above information we deduce that for $x\leq \tfrac{5}{6}\widetilde{x}_{n-1}(t)+\tfrac{1}{6}\widetilde{x}_n(t)$ and all $t_0^n\geq t_R^n$ we have \begin{align*}
x-z_n^{-R}(t)&=x-\widetilde{x}_n(t)+R+(\widetilde{x}_n(t)-z_n(t))-\big(\widetilde{x}_n(t_0^n)-z_n(t_0^n)\big)
\\ & \leq -\tfrac{5}{6}\big(\widetilde{x}_n(t)-\widetilde{x}_{n-1}(t)\big)+R+\delta_nc_n(t-t_0^n)
\\ & \leq -\tfrac{5}{3}R-\tfrac{11}{12}(c_n-c_{n-1})(t-t_0^n)+R+\tfrac{5}{8}(c_n-c_{n-1})(t-t_0^n)
\\ & \leq -\dfrac{2}{3}R-\dfrac{1}{4}(c_n-c_{n-1})(t-t_0^n).
\end{align*}
Hence, proceeding in the same fashion as before, splitting the space into two regions \[
\R=\big(-\infty,\tfrac{5}{6}\widetilde{x}_{n-1}(t)+\tfrac{1}{6}\widetilde{x}_n(t)\big]\cup\big[\tfrac{5}{6}\widetilde{x}_{n-1}(t)+\tfrac{1}{6}\widetilde{x}_n(t),+\infty\big)=:D_1\cup D_2,
\]
we deduce that for any $x\leq \tfrac{5}{6}\widetilde{x}_{n-1}(t)+\tfrac{1}{6}\widetilde{x}_n(t)$ and all $t\geq t_0^n$ we have
\[
\Psi\big(x-z_n^{-R}(t)\big)\lesssim \exp\left(-\tfrac{R}{9}-\tfrac{1}{48}(c_n-c_{n-1})(t-t_0^n)\right),
\]
Therefore, proceeding exactly in the same fashion as before, for $k=1,2,3$ we can bound \[
\mathrm{J}_k^1\lesssim \Vert u_0\Vert_{H^1}^4 e^{-\frac{R}{9}-\frac{1}{48}(c_n-c_{n-1})(t-t_0^n)}\quad \hbox{and} \quad \mathrm{J}_k^2\leq \dfrac{c_n}{2^6}\int \big(u^2+u_x^2\big)\Psi'(\cdot-z_n^R)dx.
\]
Integrating in time we obtain the desired result. The proof is complete.

\medskip

\textbf{Step 2:} Our aim now is to prove \eqref{AM_right_i_energy} in the case $i=2,...,n$. In fact, it is enough to notice that, by defining $t_i^R$ to be 
\begin{align}\label{t_i_R_def}
t_i^R:=\max\Big\{\{0\}\cup \big\{t\geq 0: \ \widetilde{x}_i(t)-\widetilde{x}_{i-1}(t)=2R\big\}\Big\},
\end{align}
we deduce that for all $t\geq t_i^R$ \[
\Vert u(t)\Vert_{L^\infty(\Pi_i(t))}\leq \dfrac{(1-\delta_i)c_i}{2^6}, \quad \hbox{where} \quad \Pi_i(t):=\left(\tfrac{5}{6}\widetilde{x}_{i-1}(t)+\tfrac{1}{6}\widetilde{x}_i(t),\widetilde{x}_i(t)-R_0\right),
\]
what in light of step $1$ is enough to prove the desired result.

\medskip

\textbf{Step 3:} Finally, in the case $i=1$ it is enough to notice that for all $t\in\R$ we have 
\[
\Vert u(t)\Vert_{L^\infty((-\infty,\widetilde{x}_1(t)-R_0))}\leq \dfrac{(1-\delta_1)c_1}{2^6}.
\]
Again, in light of step $1$ we conclude the desired result by following the same procedure. The proof is complete. \qed

\medskip

\subsection{Proof of Lemma \ref{tech_lem_left}}\label{ap_tech_left}
Let $R$ be any positive real number and consider any $t_0\in\R$ satisfying  $t_0>t_R^{i+1}$. Now we set $\widetilde{t}_0$ being \[
\widetilde{t}_0:=\max\Big\{\left\{t^{i+1}_R\right\}\cup\left\{t\in[t^{i+1}_R,t_0]: \ \widetilde{x}_i(t_0)+R+\tfrac{3}{4}\big(\widetilde{x}_i(t)-\widetilde{x}_i(t_0)\big)=y_{i+1}(t)\right\}\Big\}.
\]
Hence, by definition of $\{y_i\}_{i=1}^n$ in \eqref{def_y_i_intervals} and the definition of $\widetilde{t}_0$ above, we immediately obtain that on $\left[\widetilde{t}_0,t_0\right]$ it holds
\[
\widetilde{x}_i(t_0)+R+\tfrac{3}{4}\big(\widetilde{x}_i(t)-\widetilde{x}_i(t_0)\big)\leq y_{i+1}(t).
\]
Now we set $z_{i}^R(t):=\widetilde{x}_i(t_0)+R+\tfrac{3}{4}\big(\widetilde{x}_i(t)-\widetilde{x}_i(t_0)\big)$. Thus, by the above inequality and by using \eqref{mod_parameter_bound} we obtain that for any $t\geq \widetilde{t}_0$ and any $x\geq \tfrac{3}{8}\widetilde{x}_i(t)+\tfrac{5}{8}\widetilde{x}_{i+1}(t)$ it holds \[
x-z_{i}^R(t)\geq x-y_{i+1}(t)\geq \tfrac{1}{8}\widetilde{x}_{i+1}(t)-\tfrac{1}{8}\widetilde{x}_i(t)\geq \tfrac{R}{4}+\tfrac{1}{2^4}(c_{i+1}-c_{i})\big(t-\widetilde{t}_0\big).
\]
As before, notice that the latter inequality lead us to
\[
\Psi\left(x-z_{i}^R(t)\right)\leq \exp\left(-\frac{R}{24}-\frac{(c_{i+1}-c_i)\big(t-\widetilde{t}_0\big)}{2^7}\right),
\]
which give us the main information needed to obtain \eqref{J_one_two_ap}. Now, for the sake of simplicity let us define $\Pi_i(t):=\big(\widetilde{x}_i(t)+R,\frac{3}{8}\widetilde{x}_i(t)+\frac{5}{8}\widetilde{x}_{i+1}(t)\big)$. Hence, by the definition of both $\varepsilon_0$ and $\sigma$ in \eqref{parameters} and by using \eqref{orb_concl} we obtain 
\[
\Vert u(t)\Vert_{L^\infty\left(\Pi_i(t)\right)}<\frac{c_i}{2^7} \quad \hbox{for all }\,t\geq \widetilde{t}_0.
\]
Thus, by defining the modified energy functional \begin{align}\label{def_modified_ap}
\mathcal{I}_{i}^R(t):=\int\big(u^2+u_x^2\big)(t,x)\Psi\big(\cdot-z_{i}^R(t)\big)dx,
\end{align}
and proceeding as in Lemmas \ref{AM_orb_train} and \ref{tech_lem_mon_exp}, we deduce that for all $t\in[\widetilde{t}_0,t_0]$
\[
\mathcal{I}_{i}^R(t_0)-\mathcal{I}_{i}^R(t)\leq Ce^{-R/24}.
\]
Therefore, by the same arguments as those at the middle of Section \ref{sec_six_one} we conclude \[
\mathcal{J}_{i,r}^R(t_0)\leq \mathcal{J}_{i,r}^R(t)+Ce^{-R/24},\quad \hbox{for all } \ \widetilde{t}_0\leq t\leq t_0.
\]
Hence, it only remains to prove that the latter inequality holds for $t_R^{i+1}\leq t\leq t_0$. First of all, notice that if $\widetilde{t}_0=t_R^{i+1}$ we are done. Otherwise, by definition of $\widetilde{t}_0$ and $z_i^R$ we must have $z_i^R\big(\widetilde{t}_0\big)=y_{i+1}\big(\widetilde{t}_0\big)$. Then, if this is the case, it is enough to notice that 
\[
\widetilde{x}_{i+1}(t_R^{i+1})-y_{i+1}(t_R^{i+1})=\tfrac{1}{2}\widetilde{x}_{i+1}(t_R^{i+1})-\tfrac{1}{2}\widetilde{x}_{i}(t_R^{i+1})\geq R.
\]
Thus, by replacing $\widetilde{x}_{i+1}(t_R^{i+1})-y_{i+1}(t_R^{i+1})$ instead of $R$ in \eqref{def_modified_ap} and by redefining $z_i$ to be equals to $z_i(t)=y_{i+1}(t)$ we obtain that for all $t\in[t_R^{i+1},\widetilde{t}_0]$ it holds:
\[
\int \big(u^2+u_x^2\big)\Psi(\cdot-y_{i+1}(\widetilde{t}_0)\big)\leq \int \big(u^2+u_x^2\big)\Psi(\cdot-y_{i+1}(t)\big)+Ce^{-R/10}.
\]
Finally, since $t\geq t_R^{i+1}$, we have \[
\int \big(u^2+u_x^2\big)\Psi(\cdot-y_{i+1}(t)\big)\leq \mathcal{J}_{i,r}^R(t),
\]
from where we obtain the desired result for $t\in [t_R^{i+1},t_0]$. The proof is complete.  \qed

\medskip

\textbf{Acknowledgements :} The author is very grateful to professor Luc Molinet for encouraging me in solving this problem and for many remarkably useful conversations.


\medskip

\begin{thebibliography}{99}

\smallskip

\bibitem{AlLa} B. Alvarez-Samaniego and D. Lannes, \emph{Large time existence for 3D water-waves and asymptotics}, Invent. Math. 171 (2009), 165–186.

\bibitem{BrCo} Bressan, Alberto; Constantin, Adrian, \emph{Global dissipative solutions of the Camassa-Holm equation}. Anal. Appl. (Singap.) 5 (2007), no. 1, 1–27.

\bibitem{BrCo2} Bressan, Alberto; Constantin, Adrian, \emph{Global conservative solutions of the Camassa-Holm equation}. Arch. Ration. Mech. Anal. 183 (2007), no. 2, 215–239.

\bibitem{CH} R. Camassa, D. D. Holm, \emph{An integrable shallow water equation with peaked solitons}, Phys. Rev. Lett. 71 (1993) 1661–1664.

\bibitem{ChGuLiQu} Chen, Robin Ming; Guo, Fei; Liu, Yue; Qu, Changzheng, \emph{Analysis on the blow-up of solutions to a class of integrable peakon equations}. J. Funct. Anal. 270 (2016), no. 6, 2343–2374. 

\bibitem{Co} A. Constantin, \emph{The trajectories of particles in Stokes waves}, Invent. Math., 166 (2006), 523-535.

\bibitem{CoEs} A. Constantin and J. Escher, \emph{Global existence and blow-up for a shallow water equation}, Annali Sc. Norm. Sup. Pisa 26 (1998), 303–328. 



\bibitem{CoLa} A. Constantin and D. Lannes, \emph{The hydrodynamical relevance of the Camassa-Holm and Degasperis-Procesi Equations}, Arch. Rat. Mech. Anal., 192 (2009), 165–186.


\bibitem{CoMo} Constantin, A., Molinet, L., \emph{Global weak solutions for a shallow water equation}. Commun. Math. Phys. 211, 45–61, 2000

\bibitem{CoMo2} A. Constantin and L. Molinet, \emph{Orbital stability of solitary waves for a shallow water equation}, Phys. D 157 (2001), 75–89

\bibitem{CoSt} A. Constantin and W. A. Strauss, \emph{Stability of peakons}, Comm. Pure Appl. Math. 53 (2000), 603–610.

\bibitem{Da} Danchin, Rapha\"el, \emph{A few remarks on the Camassa-Holm equation}. Differential Integral Equations 14 (2001), no. 8, 953–988.


\bibitem{DeHoHo} Degasperis A, Holm D D and Hone A N W, \emph{A new integrable equation with peakon solution} Theor. Math.
Phys. 133 (2002) 1463–74.


\bibitem{DP} A. Degasperis, M. Procesi, \emph{Asymptotic integrability, in: A. Degasperis, G. Gaeta (Eds.), Symmetry and Perturbation Theory}, World Scientific, Singapore, 1999, pp. 23–37.

\bibitem{EM1} El Dika, Khaled; Molinet, Luc,  \emph{Exponential decay of H1-localized solutions and stability of the train of N solitary waves for the Camassa-Holm equation}. 
Philos. Trans. R. Soc. Lond. Ser. A Math. Phys. Eng. Sci. 365 (2007), no. 1858, 2313–2331.

\bibitem{EM2} El Dika, Khaled; Molinet, Luc, \emph{Stability of multipeakons}. Ann. Inst. H. Poincaré Anal. Non Linéaire 26 (2009), no. 4, 1517–1532.

\bibitem{FuFo} B. Fuchssteiner, A. S. Fokas, \emph{Symplectic structures, their B\"acklund transformations and hereditary symmetries}, Physica D4 (1981/1982) 47–66.


\bibitem{HiHo} Himonas, A. Alexandrou; Holliman, Curtis, \emph{The Cauchy problem for the Novikov equation}. Nonlinearity 25 (2012), no. 2, 449–479.

\bibitem{HoRa} Holden, Helge; Raynaud, Xavier, \emph{Global conservative multipeakon solutions of the Camassa-Holm equation}.
J. Hyperbolic Differ. Equ. 4 (2007), no. 1, 39–64.

\bibitem{HoRa2} Holden, Helge; Raynaud, Xavier, \emph{Global conservative multipeakon solutions of the Camassa-Holm equation}. J. Hyperbolic Differ. Equ. 4 (2007), no. 1, 39–64.


\bibitem{HoWa} Hone A. N. W. and Wang J, \emph{Integrable peakon equations with cubic nonlinearity} J. Phys. A: Math. Theor. 41 (2018) 372002

\bibitem{HoLuSz} Hone, Andrew N. W.; Lundmark, Hans; Szmigielski, Jacek, \emph{Explicit multipeakon solutions of Novikov's cubically nonlinear integrable Camassa-Holm type equation}. Dyn. Partial Differ. Equ. 6 (2009), no. 3, 253–289.

\bibitem{Jo} R.S. Johnson, \emph{Camassa-Holm, Korteweg-de Vries and related models for water waves}, J. Fluid Mech. 455 (2002), 63–82.

\bibitem{LiLi} Z. Lin and Y. Liu, \emph{Stability of peakons for the Degasperis-Procesi equation},
Comm. Pure Appl. Math. 62 (2009), 125–146

\bibitem{LiLiQu} Liu, Xiaochuan; Liu, Yue; Qu, Changzheng, \emph{Stability of peakons for the Novikov equation}. J. Math. Pures Appl. (9) 101 (2014), no. 2, 172–187.



\bibitem{MaMeTs} Martel, Yvan; Merle, Frank; Tsai, Tai-Peng, \emph{Stability and asymptotic stability in the energy space of the sum of N solitons for subcritical gKdV equations}. Comm. Math. Phys. 231 (2002), no. 2, 347–373. 

\bibitem{MiNo} Mikhailov V S and Novikov V S, \emph{Perturbative symmetry approach}, J. Phys. A: Math. Gen. 35 (2002) 4775–90

\bibitem{Mo} Molinet, Luc, \emph{A Liouville property with application to asymptotic stability for the Camassa-Holm equation}. Arch. Ration. Mech. Anal. 230 (2018), no. 1.

\bibitem{Mo2} Molinet, Luc, \emph{A rigidity result for the Holm-Staley b-family of equations with application to the asymptotic stability of the Degasperis-Procesi peakon}. Nonlinear Anal. Real World Appl. 50 (2019), 675–705.

\bibitem{Mo3} Molinet, Luc, \emph{Asymptotic stability for some non positive perturbations of the Camassa-Holm peakon with application to the antipeakon-peakon profile}, arXiv:1804.06230v2 

\bibitem{Mo4} Molinet, Luc, \emph{On well-posedness results for Camassa-Holm equation on the line: a survey}. J. Nonlinear Math. Phys. 11 (2004), no. 4, 521–533.


\bibitem{No}  V. Novikov, \emph{Generalizations of the Camassa-Holm type equation}, J. Phys. A42 (2009) 342002, 14pp.


\bibitem{Pa} Palacios, Jos\'e M., \emph{Asymptotic stability of peakons for the Novikov equation}, accepted in Journal of Differential Equations.

\bibitem{To} J. F. Toland, \emph{Stokes waves}, Topol. Methods Nonlinear Anal., 7 (1996), 1-48.

\bibitem{WuYi} Wu Xinglong, Yin Zhaoyang, \emph{Global weak solutions for the Novikov equation}. J. Phys. A 44 (2011), no. 5, 055202, 17 pp.


\end{thebibliography}
\end{document}